\documentclass[a4paper, 11pt, reqno]{amsart}
\usepackage{amssymb}
\usepackage{amsthm}
\usepackage{bm}
\usepackage{latexsym}
\usepackage{amsmath}
\usepackage{eufrak}
\usepackage{mathrsfs}
\usepackage{amscd}
\usepackage[all,cmtip]{xy}
\usepackage[usenames]{color}
\usepackage[dvips]{graphicx}
\usepackage{color}
\usepackage{amscd}
\usepackage{listings}
\theoremstyle{plain}

 \numberwithin{equation}{section}

\newtheorem{theorem}{Theorem}[section]
\newtheorem{proposition}[theorem]{Proposition}
\newtheorem{lemma}[theorem]{Lemma}

\newtheorem{porism}[theorem]{Porism}

\theoremstyle{definition}

\newcommand{\appsection}[1]{\let\oldthesection\thesection
\renewcommand{\thesection}{Appendix \oldthesection}
\section{#1}\let\thesection\oldthesection}

\newtheorem{definition}[theorem]{Definition}
\newtheorem{notation}[theorem]{Notation}

\theoremstyle{remark}

\newtheorem{remark}[theorem]{Remark}

\newtheorem{equations}[theorem]{Equations}

\DeclareMathOperator{\Spf}{Spf}
\DeclareMathOperator{\Pic}{Pic}

\def\oK{{\bar{K}}}

\def\Z{{\mathbb{Z}}}
\def\bS{{\mathbb{S}}}
\def\bX{{\mathbb{X}}}

\def\Q{{\mathbb{Q}}}
\def\C{{\mathbb{C}}}
\def\P{{\mathbb{P}}}
\def\A{{\mathbb{A}}}

\def\E{{\mathcal{E}}}
\def\cR{{\mathcal{R}}}

\def\m{{\mathfrak{m}}}
\def\n{{\mathfrak{n}}}
\def\cC{{\mathcal{C}}}

\def\O{{\mathcal{O}}}

\def\cH{{\mathcal{H}}}
\def\cK{{\mathcal{K}}}

\def\S{{\mathcal{S}}}
\def\X{{\mathcal{X}}}
\def\cX{{\mathcal{X}}}
\def\Y{{\mathcal{Y}}}
\def\N{{\mathcal{N}}}

\def\Z{{\mathbb{Z}}}

\def\a{{ \alpha}}
\def\b{{ \beta}}

\def\epsilon{\varepsilon}

\DeclareMathOperator{\Def}{Def}

\DeclareMathOperator{\Spec}{Spec}

\begin{document}

\title{The Craighero--Gattazzo surface is simply-connected}
\author[Julie Rana]{Julie Rana}
\email{jrana@umn.edu}
\address{School of Mathematics, University of Minnesota, 127 Vincent Hall, 206 Church St. SE, Minneapolis, MN 55455, USA.}
\author[Jenia Tevelev]{Jenia Tevelev}
\email{tevelev@math.umass.edu}
\address{Department of Mathematics and Statistics, Lederle Graduate Research Tower, 1623D, University of Massachusetts Amherst, 710 N. Pleasant Street, Amherst, MA 01003-9305, USA}
\author[Giancarlo Urz\'ua]{Giancarlo Urz\'ua}
\email{urzua@mat.uc.cl}
\address{Facultad de Matem\'aticas, Pontificia Universidad Cat\'olica de Chile, Campus San Joaqu\'in, Avenida Vicu\~na Mackenna 4860, Santiago, Chile.}
\date{\today}

\begin{abstract}
We show that the Craighero--Gattazzo surface, the minimal resolution of an explicit complex quintic surface with four elliptic singularities, is simply-connected. This was conjectured by Dolgachev and Werner, who proved that its fundamental group has a trivial profinite completion.
The Craighero--Gattazzo surface is the only explicit example of a smooth
simply-connected complex surface of geometric genus zero with ample canonical class.
We hope that our method will find other applications: to prove a topological fact about a complex surface we use an algebraic reduction mod $p$ technique and deformation theory.
\end{abstract}

\maketitle

%\tableofcontents
%----------------------------------------------------------------------------------------------------------------------------------------------
\section{Introduction}

Simply-connected minimal complex surfaces of general type of geometric genus zero, i.e.~
without global holomorphic $2$-forms,
occupy a special place in the geography of surfaces; see the excellent survey \cite{BCP}.
These surfaces are  homeomorphic (but not diffeomorphic)
to del Pezzo surfaces, i.e.~blow-ups of $\P^2$ in $9-K^2$ points where $1 \leq K^2 \leq 8$. Describing their
Gieseker moduli space of canonically polarized surfaces,
or even finding explicit examples, is difficult.
The first example was found by Barlow \cite{Ba}. Her surface has $K^2=1$ and contains four $(-2)$-curves.
Contracting them gives a canonically polarized surface with four $A_1$ singularities.
One can show by deformation theory that the local Gieseker moduli space of the Barlow surface is smooth and $8$-dimensional, and
there exist nearby surfaces which are smooth \cite[Th.~7]{CL}
and \cite{L}.

More examples, including examples for every $1\le K^2 \le 4$, were found using $\Q$-Gorenstein deformation theory, starting with the pioneering work of Lee and Park \cite{LP07}; see also \cite{PPSa,PPSb,SU}. From the moduli space perspective, the Gieseker moduli space of canonically polarized surfaces with ADE singularities is compactified by the Koll\'ar--Shepherd-Barron--Alexeev (KSBA) moduli space of canonically polarized surfaces with semi log canonical singularities \cite{KSB}. We call the complement of the Gieseker space the {\em KSBA boundary}.
Lee, Park, and others explicitly constructed special points on the KSBA boundary, and proved (using deformation theory) that the local KSBA moduli space is smooth at these points, and that one can find nearby surfaces which are smooth. To compute the fundamental group of the smoothing, one has to look into what happens when the singularity is replaced with the Milnor fiber. In the presence of special curves on the singular surface, one can use Van Kampen's theorem to compute the fundamental
group of the smoothing; see the proof of Theorem~\ref{finalresult}.

Another remarkable surface was found by Craighero and Gattazzo \cite{CG}. Their surface $S$ is the minimal resolution of singularities of an explicit quintic surface \eqref{QUIN1} with four elliptic singularities. This surface has $K_S^2=1$. It~was proved by Dolgachev and Werner \cite{DW} that $S$ is canonically polarized and that its algebraic fundamental group (i.e.~the profinite completion of the fundamental group) is trivial. In addition, it was proved by Catanese and Pignatelli \cite[Th. 0.31]{CP} that the local moduli space of $S$ is smooth of dimension~$8$. It was originally claimed in \cite{DW} that $S$ is simply-connected, but a serious flaw was
discovered in the proof; see \cite[Erratum]{DW}.

The goal of this paper is to prove that $S$ is simply-connected using
an algebraic reduction mod $p$ technique and deformation theory.
We would like to use the Lee-Park argument involving the Milnor fiber of a $\Q$-Gorenstein deformation and Van Kampen's theorem.
In order to do that, we need a $\Q$-Gorenstein family of complex surfaces $\S \to U$ over a smooth irreducible complex curve $U$,
such that one of the fibers is the Craighero--Gattazzo surface $S$ and another fiber is a simply-connected surface with
a cyclic quotient singularity and containing a special curve configuration needed to prove simply-connectedness. However, it is not clear how to
explicitly construct a family containing the Craighero--Gattazzo as a fiber because no explicit model of the moduli space is known.

Our trick is to work out an integral model of the Craighero--Gattazzo surface over a ring of algebraic integers. One obvious model is given by the quintic equation. In an REU (research experience for undergraduates) directed by the first two authors, Charles Boyd discovered that this arithmetic threefold has a non-reduced fiber in characteristic $7$, and its local equation has a very special form.
Over the complex disc, analogous families of quintic surfaces were studied by the first author in~\cite{R}, where it was proved that the KSBA replacement acquires a ${1\over 4}(1,1)$ singularity in the special fiber. In fact, it is proved in~\cite{R} that
numerical quintic surfaces with a ${1\over 4}(1,1)$ singularity form a divisor in the KSBA moduli space (and this divisor is explicitly described). The upshot is that, to some degree, it can be hoped that this singularity appears in one-parameter families of surfaces, including families over a ring of algebraic integers.
We show that the KSBA limit of $S$
over the $7$-adic disc is a surface $S_0$ with a ${1\over 4}(1,1)$ singularity.
We use the word ``KSBA limit'' somewhat loosely here because existence of the mixed characteristic
KSBA moduli space (or even canonical KSBA integral models) is still only conjectural.

The minimal resolution of $S_0$ turns out to be a very special and beautiful Dolgachev surface, i.e.~an elliptic fibration over $\P^1$ with two multiple fibers, one of multiplicity $2$ and one of multiplicity $3$. We call it the Boyd surface.
By pure luck, it carries a special curve, which, if it were a complex surface, would have allowed us to conclude that the Craighero--Gattazzo surface $S$ is simply connected.
Of course our degeneration is over the $7$-adic unit disc, so we can not use Van Kampen's theorem directly.
Our main idea is to use deformation theory to conclude that $S$ admits an analogous (but no longer explicit) degeneration over the complex unit disc to a complex surface $D_0$ with a ${1\over 4}(1,1)$ singularity such that its minimal resolution is a complex Dolgachev surface analogous to the Boyd surface.

As an application of our construction, we show in Theorem~\ref{Lefschetz} that
there exist simply-connected Dolgachev surfaces (with multiple fibers of multiplicity~$2,3$)
which carry algebraic genus~$2$ Lefschetz fibrations, specifically
genus~$2$ fibrations without multiple components in fibers and such that the only singularities of fibers are nodes. Dolgachev and Werner showed existence of a genus~$2$ fibration on the Craighero--Gattazzo surface~\cite[Prop.3.2]{DW}. If this fibration had only nodal singular fibers, then by combining our theorem that the Craighero--Gattazzo surface is simply-connected, we would have the existence of a simply connected numerical Godeaux surface with a genus $2$ Lefschetz fibration.
By \cite{Fr}, these surfaces are homeomorphic to $\P^2$ blown-up in $9$ or $8$ points, respectively.
In the symplectic category,  Lefschetz fibrations on knot
surgered elliptic surfaces in the homotopy class of $\P^2$ blown-up at $9$ points were constructed in \cite{FS}
and in the homotopy classes of $\P^2$ blown-up at $8$ or $7$ points in \cite{BK}.

\subsection*{Acknowledgements}

We are grateful to Paul Hacking for numerous discussions about moduli of stable surfaces,
to Inanc Baykur for his suggestion to construct Lefschetz fibrations mentioned above and to Charles Boyd for writing and testing Macaulay2 scripts which were used to find the KSBA limit of the Craighero--Gattazzo surface in characteristic $7$. The first author was partially supported by the NSF grant DMS-1502154. The second author was supported by the NSF grant DMS-1303415. The third author was supported by the FONDECYT regular grant 1150068.

%-------------------------------------------------------------------------------------------------------------------

\section{Stable limit of the CG surface in characteristic~$7$}

\bigskip

Let $X\subset \P^3_\C$ be the quintic surface
\smallskip
$$a^{2}(x^{2} y^{3}+x^{3}t^{2}+y^{2} z^{3}+z^{2}t^{3})+
m^{2}( x^{3} z^{2}+x^{2} z^{3}+y^{3} t^{2}+y^{2} t^{3})+\qquad\qquad{}$$
$$2 a m (x yz^{3}+x y^{3} t+x^{3} z t+y z t^{3})+
14 m (x^{3} y z+y^{3} z t+x z^{3} t+x y t^{3})+\qquad\qquad{}$$
\begin{equation}\label{QUIN1}
7b (x^{2} y^{2} z+y^{2} z^{2} t+x^{2} y t^{2}+x z^{2} t^{2})+
14 a( x y^{3} z+x^{3} y t+y z^{3} t+x zt^{3})+ \ \ \ \ \ \ \ \ \ \ \ \ \ \
\end{equation}
$$c (x^{2} y z^{2}+x^{2} z^{2} t+x y^{2} t^{2}+y^{2} z t^{2})+
7e (xy^{2} z^{2}+x^{2} y^{2} t+x^{2} z t^{2}+y z^{2} t^{2})+\qquad\qquad{}$$
$$f (x^{2} y z t+x y^{2} z t+x y z^{2} t+x y z t^{2})+
49(x^{3} y^{2}+y^{3} z^{2}+z^{3} t^{2}+x^{2} t^{3})=0.\qquad\qquad{}$$
\smallskip

\begin{figure}[htbp]
\includegraphics[width=8cm]{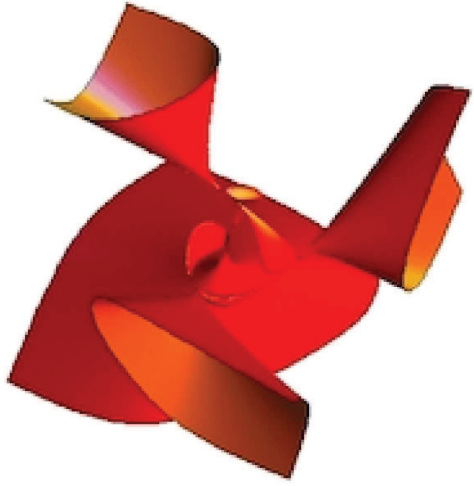}
\caption{The Craighero--Gattazzo quintic}
\end{figure}

\noindent
The coefficients are (from \cite[page 25]{CP}, multiplied by $49$)
$$a=7r^{2},\quad b=-2r^{2}+13r+18,\quad c=73r^{2}+75r+92,$$
$$e=-r^{2}+24r+9,\quad f=181r^{2}+241r+163,\quad m=3r^{2}+5r+1,$$
\smallskip
where $r$ is a complex root of the equation
\begin{equation}\label{cubeq}
r^3+r^2-1=0.
\end{equation}
The surface is invariant under the $\mu_4$ action which cyclically permutes the variables
as follows: $x\to y\to z\to t\to x$. It is singular at the points $$P_1=[1:0:0:0], \ P_2=[0:1:0:0], \ P_3=[0:0:1:0], \ P_4=[0:0:0:1].$$ Its minimal resolution is the {\em Craighero--Gattazzo surface~$S$}. Exceptional divisors over $P_1,\ldots,P_4$ are elliptic curves $\E_1,\ldots,\E_4$ such that $\E_i^2=-1$ for each~$i$.
These singularities are sometimes called {\em singularities of type $\tilde E_8$}.

The equation \eqref{QUIN1} gives an integral model of $X$ over $\Spec\Z[r]$. Since $3$ is a simple root of \eqref{cubeq} in $\Z/(7)$, by Hensel's Lemma we have a section $\Spec\Z_7\to\Spec\Z_7[r]$,
where $\Z_7$ is the ring of $7$-adic integers.
 Pulling back the integral model with respect to the base change $\Spec\Z_7\to\Spec\Z_7[r]\to\Spec\Z[r]$ gives the family
$\X$ over $\Spec\Z_7$. The corresponding root of \eqref{cubeq} modulo $7^3$ is equal to $143$ and
after some manipulations the equation of $\X$ to the order of $7^3$ takes the form
\begin{equation}\label{ranafamily}
f_1f_2^2 + 7f_2f_3 + 7^2f_5+\hbox{\rm (higher order terms)},
\end{equation}
where $f_1,f_2,f_3,f_5\in\Z/(7)[x,y,z,t]$ are the following forms (the subscript indicates the degree):
$$f_1=x+y+z+t,$$
$$f_2=x z+y t,$$
$$f_3=
2( x^{2} y
+y^{2} z
+ z^{2} t
+ x t^{2})
+x^{2} z
+x z^{2}
+y^{2} t
+y t^{2}-$$
$$3 (x y^{2}
+y z^{2}
+x^{2}t
+z t^{2}
+x y z
+x y t
+x z t
+y z t),$$
and
$$f_5=
x^{3} y^{2}
+x^{3}z^{2}
+y^{3} z^{2}
+x^{2} z^{3}
+y^{3} t^{2}
+z^{3} t^{2}
+x^{2} t^{3}
+y^{2} t^{3}
+$$
$$x^{3} y z
+y^{3} z t
+x z^{3} t
+x yt^{3}
-x y^{2} z^{2}
-x^{2} y^{2} t
-x^{2} z t^{2}
-y z^{2} t^{2}-$$
$$x^{2} y z t
-x y^{2} z t
-x y z^{2} t
-x y z t^{2}-
-3 x^{2} y^{3}
-3 y^{2}z^{3}
-3 x^{3} t^{2}
-3 z^{2} t^{3}
-$$
$$2 x^{2} y^{2} z
-2 x^{2} y z^{2}
-2 x^{2}z^{2} t
-2 y^{2} z^{2} t
-2 x^{2} y t^{2}
-2 x y^{2} t^{2}
-2 y^{2} z t^{2}
-2 x z^{2} t^{2}
-$$
$$3 x y^{3} z
-3 x^{3} y t
-3 y z^{3} t
-3 x z t^{3}.$$

This expansion  shows that the special fiber of $\X$  is
the union of the plane $L=(f_1=0)$ and
the quadric surface $Q=(f_2=0)$ with multiplicity $2$.
In~particular, it is not reduced.

Let $k$ be an algebraically closed field of characteristic $7$
and let $\cR$ be its ring of Witt vectors.
We denote the pull-back of $\X$ to $\Spec\cR$ (with respect to the canonical inclusion $\Z_7\hookrightarrow\cR$)
by the same letter~$\X$. We also pullback $L$ and $Q$ to $k$.

We would like to compute the stable limit of the generic fiber of $\X$. Over the complex disc, stable $\Q$-Gorenstein limits of families of the form \eqref{ranafamily} were computed by the first author \cite{R}, and semi-stable Gorenstein limits
of sufficiently general families
by Ashikaga and Konno~\cite{AK}. In our case the disc is $7$-adic but the computation is the same.
We now describe what the stable limit is, postponing the proof to Lemma~\ref{mainboydthing}.

Let $\Delta=L\cap Q\subset Q\simeq\P^1_k\times\P^1_k$. It is a curve in the linear system $|\O(1,1)|$.
The curve $$Q\cap (f_3^2-4f_1f_5=0)\subset\P^1_k\times\P^1_k$$
is the union of two curves in the linear system $|\O(3,3)|$:
\begin{equation}\label{b1eqqqq}
B_1=Q\cap(x y^{2}+3 x^{2} z-3 y^{2} z+3 x z^{2}-3
      x t^{2}+z t^{2}=0)
\end{equation}
and
\begin{equation}\label{b2eqqqq}
B_2=Q\cap(yz^2+3y^2t-3z^2t+3yt^2-3yx^2+tx^2=0).
\end{equation}

Figure~\ref{f1} shows how these curves intersect, where
$A_1,\ldots,A_4$ are rulings of $\P^1_k\times\P^1_k$ and $\{Q_1,Q_2\}=\Delta\cap B_1\cap B_2$.

\begin{figure}[htbp]
\includegraphics[width=10cm]{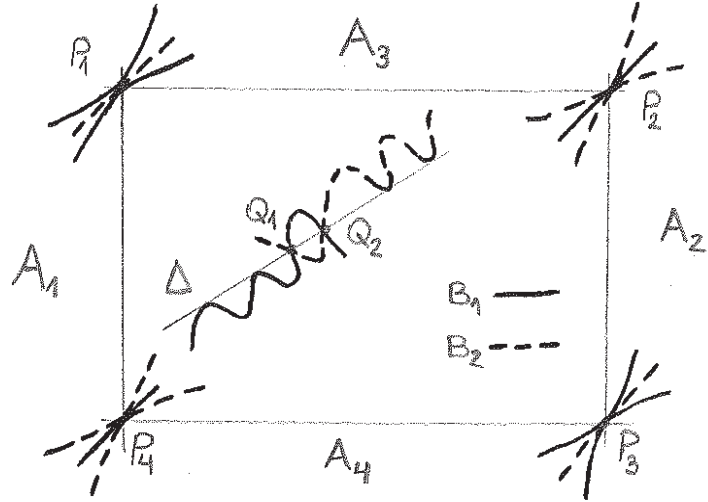}
\caption{Data in $Q\simeq\P_k^1 \times \P_k^1$}
\label{f1}
\end{figure}

\begin{lemma}
Let $$\pi:\,Z\to \P^1_k\times\P^1_k$$ be the double cover branched along $B_1\cup B_2$.
The surface $Z$ has $4$ simple elliptic singularities of type $\tilde E_8$
over $P_1,\ldots,P_4$, and two $A_1$ singularities over $Q_1$ and~$Q_2$. It is smooth elsewhere.
\end{lemma}

\begin{proof}
Direct calculation.
\end{proof}

We denote the ramification curves in $Z$ by $B_1$ and $B_2$, and we denote the singular points of $Z$ by the same letters as their images in $\P^1\times\P^1$.
Finally, $\pi^{-1}(\Delta)$ is the union of two smooth rational curves: $\Delta_1$ and $\Delta_2$.

Unless it causes confusion, we adopt the following convention throughout this paper: we use the same letter to denote an irreducible curve and its proper transform
after some birational transformation.

\begin{definition}
We call the minimal resolution $Y$ of $Z$ \emph{the Boyd surface}.
\end{definition}

The Boyd surface contains elliptic curves $E_1,\ldots,E_4$ of self-intersection~$-1$ (preimages of elliptic singularities of $Z$),
$(-2)$-curves $N_1$ and $N_2$ (preimages of $A_1$ singularities of $Z$), and $(-4)$-curves $\Delta_1$ and $\Delta_2$.

\begin{definition}
 Let $S_0$ be the surface obtained by contracting the $(-4)$-curve $\Delta_1$.
\end{definition}

\begin{lemma}[{c.f. \cite{R}}]\label{mainboydthing}\label{asfgasrgarg}
There exists a flat family $S\to\Spec\cR$ with special fiber~$S_0$
and generic fiber the Craighero--Gattazzo surface $S$ (after pull-back to $\C$).
Near the singular point of the special fiber, the family is formally isomorphic to
$$(xy=z^2+7)\subset{1\over2}(1,1,1)_\cR:=\Spec\cR[x,y,z]^{\mu_2},$$
where $\mu_2$ acts by $x\mapsto -x$, $y\mapsto -y$, $z\mapsto -z$.
\end{lemma}

\begin{proof}
We first produce the stable limit of the Craighero--Gattazzo quintic~$X$ in characteristic $7$.
Let $\X^0$ be the generic fiber of $\X$ given by equations \eqref{ranafamily}.
Consider the family $\hat\X\to\Spec\cR$ given by equations
$$(f_1w^2+f_3w+f_5+h.o.t.=0,\quad f_2=7w)\subset\P^4_{[x:y:z:t:w]}(1,1,1,1,2)_\cR$$
obtained by substituting $f_2$ for $7w$ in the first three terms of \eqref{ranafamily} and dividing by~$343$. Here, and throughout, ``h.o.t." refers to higher order terms with respect to the $7$-adic valuation.
The generic fiber of $\hat\X$ is clearly isomorphic to $\X^0$.

The special fiber $\hat\X _0$ is given by equations
$$(f_1w^2+f_3w+f_5=0,\quad f_2=0)\subset\P^4_{[x:y:z:t:w]}(1,1,1,1,2)_k.$$
We claim that it is isomorphic to the surface $Z'$ obtained by blowing down $4$ elliptic $(-1)$-curves on $S_0$
to $\tilde E_8$-singularities.

The point $(0:0:0:0:1)$ is an isolated singularity with equation, in a local chart,
$$(f_1+f_3+f_5=0,\quad f_2=0)\subset{1\over2}(1,1,1,1).$$
The singularity is formally isomorphic to
$$(xy=z^2)\subset{1\over2}(1,1,1)_k,$$
which has a $(-4)$-curve as the resolution graph.
Moreover, the equation of the whole family $\hat\X$ near this point is formally isomorphic to
$$(xy=z^2+7)\subset{1\over2}(1,1,1)_\cR.$$

Next we analyze $\hat\X _0$ away from $t_0=(0:0:0:0:1)$.
We use the generically $2:1$ map $\pi:\,S_0\setminus\{t_0\}\to Q$ given by $[x:y:z:t:w]\to [x:y:z:t]$.
Away from $\Delta=L\cap Q$, $\pi$~is a double cover branched along $(f_3^2-4f_1f_5=0)=B_1\cup B_2$.
Thus it can be identified with $Z'\setminus (\Delta_2\cup N_1\cup N_2)$.
Over $\Delta$, but away from $t_0$ (which includes $Q_1$ and $Q_2$) the map $\pi$ is one-to-one.
The preimages of $Q_1$ and $Q_2$ are lines (with coordinate $w$).
The preimages of the other four points where  $f_3=0$ are empty; in Figure~\ref{f1} these are the points where $B_1$ and $B_2$ are tangent to $\Delta$.
It follows that $\hat\X _0$ and $Z'$ are normal surfaces isomorphic in codimension~$1$, and therefore isomorphic.

It remains to notice that the family $\hat\X$ has $\tilde E_8$ singularities
along the sections $(1:0:0:0:0)$, $(0:1:0:0:0)$, $(0:0:1:0:0)$, and $(0:0:0:1:0)$.
Resolving them gives a family $\S\to\Spec\cR$ with special fiber $S_0$ and
generic fiber (after pulling back to $\Spec\C$) the Craighero--Gattazzo surface~$S$.
\end{proof}

%-------------------------------------------------------------------------------------------------------------
\section{Study of the Boyd surface - vanishing of obstructions}\label{boydgeo1}

We have a commutative diagram,
$$\begin{CD}
W @>\tau>> Y @>>> Z\\
@V{\pi'}VV && @VV{\pi}V\\
\P &@>\sigma>> & Q=\P_k^1\times\P_k^1\\
\end{CD}$$
where the vertical maps are double covers and the horizontal maps are birational. Here $\P$ is obtained by blowing up $Q_1$ and $Q_2$
(let $\bar N_1$ and $\bar N_2$ be the exceptional divisors), blowing up $P_1,\ldots,P_4$ (let $\bar G_1,\ldots,\bar G_4$
be the exceptional divisors), and then blowing up these $4$ points again in the direction of the tangent cone to $B_1\cup B_2$
(let $\bar E_1,\ldots,\bar E_4$ be the exceptional divisors).

Since $$B_1 + B_2 + 2\bar N_1 + 2\bar N_2
\sim 6\sigma^*(\O_Q(1,1))-3 \sum_{i=1}^4 \bar G_i -6 \sum_{i=1}^4\bar E_i,$$ we have
\begin{equation}\label{sdghfghfgh}
B_1 + B_2 + 2\bar N_1 + 2\bar N_2
\sim  3 \left(\sum_{i=1}^4 A_i + \sum_{i=1}^4 \bar G_i\right)
\end{equation}
as well as
\begin{equation}
 B_1 + B_2 + \sum_{i=1}^4\bar G_i \sim 2 \Big( 3 \sigma^*(\O_Q(1,1))- \bar N_1-\bar N_2-3\sum_{i=1}^4\bar E_i - \sum_{i=1}^4\bar G_i\Big).
 \label{jhgjhgjhghjgj}\end{equation}
We define $W$ to be the double cover of $\P$ branched along the smooth curve
$$B=B_1+B_2+\bar G_1+\ldots+\bar G_4.$$
Let $N_i, E_i, G_i\subset W$ be the preimages of $\bar N_i, \bar E_i, \bar G_i$, respectively. The curves
$G_1,\ldots,G_4$ are $(-1)$-curves, and contracting them gives the Boyd surface $Y$. The curves $N_1$ and $N_2$ are $(-2)$-curves on $Y$, while
$E_1,\ldots,E_4$ are elliptic $(-1)$-curves (i.e.~elliptic curves with self-intersection $-1$).

\begin{theorem}\label{julie}
$H^2(Y, T_Y(-\log(\Delta_1+N_1)))=0$.
\end{theorem}

\begin{proof}
We follow \cite[4.8, 4.10]{R} closely. It suffices to show that
\begin{equation}
H^2(W, T_W(-\log(\Delta_1+N_1)))=0.
\label{asrgadfhdfhdfh}\end{equation}
Indeed, if this is the case then Serre duality implies
$$0=H^0(W, \Omega^1_W(\log(\Delta_1+N_1))(K_W))=$$
$$H^0(Y, \tau_*\left[\Omega^1_W(\log(\Delta_1+N_1))(G_1+\ldots+G_4)\right](K_Y))=$$
(by Lemma~\ref{pushforward})
$$=H^0(Y, \Omega^1_Y(\log(\Delta_1+N_1))(K_Y))=H^2(Y, T_Y(-\log(\Delta_1+N_1)))^\vee.$$

Arguing as in \cite[4.8]{R}, \eqref{asrgadfhdfhdfh} will follow if we can show that
\begin{equation}\label{1van}
H^2(W, T_W(-\log(\Delta_1+\Delta_2+N_1)))_-=0,
\end{equation}
and
\begin{equation}\label{3van}
H^2(W, T_W(-\log(N_1)))_+=0,
\end{equation}
where $+/-$ denotes the symmetric/skew-symmetric part with respect to the $\mu_2$-action on the double cover.
Explicitly, and using Serre duality multiple times, if $\alpha\in H^0(W, \Omega^1_W(\log(\Delta_1+N_1))(K))$, then since
$$\Omega^1_W(\log(\Delta_1+N_1))(K)\subset \Omega^1_W(\log(\Delta_1+\Delta_2+N_1))(K)$$
the one-form $\alpha$ must be invariant. But $\mu_2$ interchanges $\Delta_1$ and $\Delta_2$, so that $\alpha$ does not have a pole along $\Delta_1$.
Thus, $\alpha\in \Omega_W^1(\log N_1)(K)$ is an invariant one-form. Equation \eqref{3van} implies that $\alpha=0$.

\underline{\emph{Proof of \eqref{1van}.}}
At each of the points $Q_3, \ldots, Q_6$ (the remaining points of $B_i\cap \Delta$) we blow up twice to obtain a surface $\P_1$ where $\Delta$ and $B_i$ have normal crossings. Let $\bar C_i$, $\bar F_i$, $i=3,\ldots, 6$ be the exceptional divisors of these blowups, so that on $\P_1$ we have $\bar C_i^2=-2$ and $\bar F_i^2=-1$.
Let $\sigma' \colon \P_1 \to Q$ be the composition of these blowups, and let $f:\,W_1\to\P_1$ be the double cover branched over $B_1+B_2+\sum\bar G_i+\sum\bar C_i$.

The surface $W_1$ contains $(-1)$-curves $C_i$ and $(-2)$-curves $F_i$ which contract to give the surface $W$. By the $(-1)$- and $(-2)$-curve principles~\cite[Prop. 4.3, Thm. 4.4]{PSU}
(here we only need the $(-1)$-curve principle), we have
$$H^2(W_1, T_{W_1}(-\log(\Delta_1+\Delta_2+N_1)))\simeq H^2(W, T_W(-\log(\Delta_1+\Delta_2+N_1))).$$

Notice that the double cover $f$ is defined by (see \eqref{jhgjhgjhghjgj})
$$B_1+B_2+\sum\bar G_i+\sum \bar C_i \sim 2 L,$$
where
$$L\sim 3\sigma'^*(\O_Q(1,1))-\sum\bar G_i-3\sum \bar E_i-\sum \bar N_i-\sum \bar F_i.$$
Also we have
$$K_{\P_1}=-2\sigma'^*(\O_Q(1,1))+\sum \bar N_i +\sum \bar G_i +2\sum \bar E_i+\sum\bar C_i+2\sum\bar F_i,$$
and so
$$K_{\P_1}+L\sim \sigma'^*(\O_Q(1,1))-\sum\bar E_i+\sum\bar C_i+\sum\bar F_i.$$

By Lemma~\ref{double cover}, we have
$$f_*(T_{W_1}(-\log(\Delta_1+\Delta_2+N_1)))_-=T_{\P_1}(-\log(\Delta+\bar N_1))(-L).$$
By Serre duality, it suffices to prove vanishing of
$$H^0(\P_1, \Omega_{\P_1}^1(\log(\Delta+\bar N_1))(K_{\P_1}+L)),$$
or
$$H^0(\P_1, \Omega_{\P_1}^1(\log(\Delta+\bar N_1))(\sigma'^*(\O_Q(1,1))-\sum\bar E_i+\sum\bar C_i+\sum\bar F_i).$$

By Lemma~\ref{pushforward}, we have
\begin{eqnarray*}
\sigma'_*(\Omega_{\P_1}^1(\log(\Delta+ \bar N_1))(\sigma'^*(\O_Q(1,1))-\sum\bar E_i+\sum\bar C_i+\sum\bar F_i))\hspace{.5in}&\\
\subset
\sigma'_*( \Omega_{\P_1}^1(\Delta+ \bar N_1 +\sigma'^*(\O_Q(1,1))-\sum\bar E_i+\sum\bar C_i+\sum\bar F_i))&\\
=\sigma'_*( \Omega_{\P_1}^1(\sigma'^*(\Delta)+\sigma'^*(\O_Q(1,1))- \bar N_2 -\sum\bar E_i-\sum\bar F_i))&\\
\subset \Omega_Q^1 \otimes \O_Q(2,2)\otimes\mathcal{I}_{Q_2}\bigotimes_{i=1}^4 \mathcal{I}_{P_i}&\\
\end{eqnarray*}
Since
$\Omega_Q^1=\O_Q(-2, 0)\oplus \O_Q(0,-2)$, we have
$$\Omega_Q^1 \otimes \O_Q(2,2)=\O_Q(0,2)\oplus \O_Q(2,0).$$
Thus, any global section of $\Omega_Q^1 \otimes \O_Q(2,2)\otimes\mathcal{I}_{Q_2}\bigotimes_{i=1}^4 \mathcal{I}_{P_i}$ is a global section of $\O_Q(0,2)\oplus \O_Q(2,0)$ vanishing at the points $Q_2, P_1, \ldots, P_4$. Since these points are in three distinct horizontal and vertical fibers of $Q$, any such global section must be $0$. This completes the proof of Equation~\ref{1van}.

%Since the only global section of $\O_Q(0,2)\oplus \O_Q(2,0)$ vanishing at the points $P_1, \ldots, P_4$ corresponds to the divisor $\sum A_i$, which does not pass through $Q_2$, such a section does not exist.

\begin{lemma}\label{double cover}
Let $Y$ be a smooth projective surface defined over an algebraically closed field of characteristic $\ne2$.
Let $f:X\rightarrow Y$ be a double cover with a
smooth branch divisor $B\subset Y$. Let $C=f^{-1}(D)$ be the preimage of a smooth curve $D$ on $Y$, and suppose that $D$ intersects $B$ transversally. Then
$$f_*(\Omega_X^1(\log C))=\Omega_Y^1(\log(D))\oplus\Omega_Y^1(\log (D+B))(-L)$$
and
$$f_*(T_X(-\log C))=T_Y(-\log(D+B))\oplus T_Y(-\log(D))(-L) $$
where $B\sim 2L$.
Moreover, these decompositions break the sheaves into their invariant and anti-invariant subspaces under the action of $\mu_2$ by deck transformations.
\end{lemma}
\begin{proof}
The surface $X$ is defined in the total space of the line bundle $L$ by the equation $z^2=x$ where $x$ is a global section of
$\mathcal{O}_{Y}(2L)$. This allows us to work \'etale-locally, using the argument of~\cite[4.6]{R}.
\end{proof}

%\bigskip
%$\clubsuit$ EXPAND: Why don't we just use ~\cite[4.6]{R} with local parameters?
%\bigskip

\begin{lemma}\label{pushforward} Let $Y$ be a smooth
projective surface defined over an algebraically closed field. Let $\sigma: X\rightarrow Y$ be the blowup of $p\in Y$
with exceptional divisor~$E$. Then for every integer $m\ge 0$, we have $\sigma_*(\Omega_X^1(mE))= \Omega_Y^1$. Moreover,
$\sigma_*(\Omega_X^1(-E))=\Omega_Y^1\otimes \mathcal{I}_p$, where $\mathcal{I}_p$ is the ideal sheaf of the point $p$.
\end{lemma}

\begin{proof}
Let $\eta$ be the generic point of $Y$. The sheaves $\sigma_*(\Omega_X^1(mE))$ and $\Omega_Y^1$ are subsheaves of the
constant sheaf with stalk $\Omega_{Y,\eta}^1$ (the sheaf of rational differentials).
A local section of $\sigma_*(\Omega_X^1(mE))$ is regular outside of~$p$ and therefore regular at $p$
since $\Omega_Y^1$ is locally free. Thus we have an injective map
$i: \sigma_*(\Omega_X^1(mE))\rightarrow \Omega_Y^1$.
 It is surjective because given a local $1$-form $\alpha\in\Omega_Y^1(U)$, the $1$-form $\sigma^*(\alpha)\in\Omega_X^1(\sigma^{-1}(U))\subset\Omega_X^1(mE)(\sigma^{-1}(U))$ maps to $\alpha$.

For the second part, we have an injective map $i: \sigma_*(\Omega_X^1(-E))\rightarrow \Omega_Y^1$, as above. Moreover, any one-form $i(\alpha)$ in the image of $i$ vanishes at $p$, since $\alpha$ vanishes along $E$. Thus, the image of $i$ is the sheaf $\Omega_Y^1\otimes \mathcal{I}_p$.
\end{proof}

\underline{\emph{Proof of \eqref{3van}.}}
Note that we have the short exact sequence
$$0\rightarrow T_W(-\log(\Delta_1+\Delta_2+N_1))\rightarrow T_W(-\log N_1)\rightarrow \N_{\Delta_1/W}\oplus \N_{\Delta_2/W}\rightarrow 0.$$
Since $H^2(W, \N_{\Delta_1/W}\oplus \N_{\Delta_2/W})=0$, it suffices to prove that
$$H^2(W, T_W(-\log(\Delta_1+\Delta_2+N_1)))_+=0.$$
This part is more delicate and  the proof occupies the rest of the section.

%At each of the points $Q_3, \ldots, Q_6$ (the remaining points of $B_i\cap \Delta$) we blow up twice to obtain a surface $\P_1$ where $\Delta$ and $B_i$ have normal %crossings. Let $\bar C_i$, $\bar F_i$, $i=3,\ldots, 6$ be the exceptional divisors of these blowups, so that on $\P_1$ we have $\bar C_i^2=-2$ and $\bar F_i^2=-1$.
%Let $f:\,W_1\to\P_1$ be the double cover branched over $B_1+B_2+\sum\bar G_i+\sum\bar C_i$.

%The surface $W_1$ contains $(-1)$-curves $C_i$ and $(-2)$-curves $F_i$ which contract to give the surface $W$. By the $(-1)$- and $(-2)$-curve principles~\cite[Prop. 4.3, Thm. 4.4]{PSU}
%(here we only need the $(-1)$-curve principle), we have
%$$H^2(W_1, T_{W_1}(-\log(\Delta_1+\Delta_2+N_1)))\simeq H^2(W, T_W(-\log(\Delta_1+\Delta_2+N_1))).$$
By Lemma~\ref{double cover}, we have
$$f_*(T_{W_1}(-\log(\Delta_1+\Delta_2+N_1)))_+=T_{\P_1}(-\log(\Delta+\bar N_1+B_1+B_2+\sum \bar G_i+\sum\bar C_i)).$$
Again applying the $(-1)$ and $(-2)$-curve principles, it suffices to show that
$$H^2(\P_1, T_{\P_1}(-\log(\Delta+B_1+B_2)))=0.$$

To begin with, we claim that
\begin{equation}\label{onp1}
H^2(\P_1, T_{\P_1}(-\log(B_1+B_2)))=0.
\end{equation}
Because $B_1$ and $B_2$ have simple normal crossings after contracting the curves $\bar C_i$ and $\bar F_i$, it suffices to show that
$$H^2(\P, T_{\P}(-\log(B_1+B_2)))=0$$
or equivalently (by Serre duality)
$$ H^0(\P, \Omega_{\P}^1(\log(B_1+B_2))(-2D+\sum \bar N_i+\sum\bar G_i+2\sum\bar E_i))=0.$$

Letting $\mathcal{F}=\O_{\P}(-2D+\sum \bar N_i+\sum\bar G_i+2\sum\bar E_i)$,
we have the short exact sequence
$$0\rightarrow\Omega_{\P_1}^1\otimes\mathcal{F}\rightarrow \Omega_{\P}^1(\log(B_1+B_2))\otimes\mathcal{F}\rightarrow
(\O_{B_1}\oplus\O_{B_2})\otimes \mathcal{F}\rightarrow 0.$$
The products $B_j\cdot \mathcal{F}=-4<0$ for $j=1, 2$ and thus
$$H^0((\O_{B_1}\oplus\O_{B_2})\otimes \mathcal{F})=0.$$
The projection formula and Lemma~\ref{pushforward} give
$$H^0(\P_1,\Omega_{\P_1}^1\otimes\mathcal{F})\simeq H^0(Q, \Omega_Q^1(-2D)).$$
The sheaf $\Omega_Q^1(-2D)=\O_Q(-4,-2)\oplus\O_Q(-2,-4)$ has no global sections, completing the  proof of claim~\eqref{onp1}.

Now consider the short exact sequence
$$0\rightarrow T_{\P_1}(-\log(\Delta+B_1+B_2))\rightarrow T_{\P_1}(-\log(B_1+B_2))\rightarrow \N_{\Delta/\P_1}\rightarrow 0.$$
By claim~\eqref{onp1}, vanishing of $H^2(\P_1, T_{\P_1}(-\log(\Delta+B_1+B_2)))$ will be complete once we show that the map
\begin{equation}\label{issurjective?}
H^1(\P_1,T_{\P_1}(-\log(B_1+B_2)))\rightarrow H^1(\P_1,\N_{\Delta/\P_1})
\end{equation}
is surjective.  We identify $H^1(\P_1,T_{\P_1}(-\log(B_1+B_2)))$ with  the space of first-order infinitesimal deformations of $\P_1$
which contain an embedded first-order deformation of $B_1\cup B_2$. We identify $H^1(\P_1,\N_{\Delta/\P_1})$ with the space of obstructions to deforming $\Delta$ in $\P_1$. Thus, the map~\eqref{issurjective?} factors through the natural map
\begin{equation}\label{qqqyyyy}
H^1(\P_1,T_{\P_1})\rightarrow H^1(\P_1,\N_{\Delta/\P_1})
\end{equation}
which sends an infinitesimal first-order deformation of $\P_1$ to the obstruction to deforming $\Delta$ in this
first-order deformation of $\P_1$.
We have to show that given any such obstruction, there is a deformation of
the pair $(\P_1, B_1+B_2)$ that maps to the given obstruction.

%The deformation space of $\P_1$ is isomorphic to the space of deformations of $\P^1\times\P^1$

%We are going to do it by an explicit calculation
%in the spirit of Zariski--Wahl theory of equisingular deformations - $\clubsuit$ EXPAND (CORE).

Recall that $\P_1$ is obtained from $\P^1 \times \P^1$ by blowing up once at each of $Q_1, \ldots, Q_6; P_1,\ldots,P_4$,
and again at each of $Q_3, \ldots, Q_6$ in the direction of the proper transform of $\Delta$
and at each of $P_1, \ldots, P_4$ in the direction of tangent cone of $B_1\cup B_2$.
We denote by $\sigma_2 \colon \P_1 \to \tilde{Q}$ the ``intermediate'' blowup, i.e. the map which contracts the last eight $(-1)$-curves on $\P_1$.

We have the following exact sequence of sheaves on $\tilde{Q}$ %(see, e.g., Hartshorne Deformation Theory exercise 10.5?)
$$0\rightarrow (\sigma_2)_*T_{\P_1}\rightarrow T_{\tilde{Q}}\rightarrow \bigoplus_{i=3}^6 k_{Q_i}^2 \oplus\bigoplus_{i=1}^4 k_{P_i}^2\rightarrow 0.$$
Looking at the corresponding exact sequence in cohomology, we see that every infinitesimal first-order deformation of $\P_1$ arises from either an infinitesimal first-order deformation of $\tilde{Q}$ (corresponding to an element of $H^1(\tilde{Q},T_{\tilde{Q}})$)  or from an infinitesimal first-order deformation of the points $Q_3, \ldots, Q_6, P_1,\ldots,P_4$ on $\tilde{Q}$, or both. This latter space is isomorphic to a vector space $V=(k^2)^8$. We note that $V$ has a linear subspace $V_1\simeq k^8$
corresponding to infinitesimal first-order deformations of the points $Q_3, \ldots, Q_6$, $ P_1,\ldots,P_4$ to points along the exceptional divisors of $\sigma_2$, i.e. changing the tangent direction of the infinitely-near blowup.

Similarly, because $\P^1\times \P^1$ is rigid, every first-order infinitesimal deformation of $\tilde{Q}$ arises from a first-order infinitesimal deformation of the points $Q_1, \ldots, Q_6; P_1,\ldots,P_4$ in $\P^1\times\P^1$. This latter deformation space is isomorphic to the vector space $W=(k^2)^{10}$. Thus, we have short exact sequences
$$0\to V\rightarrow H^1(\P_1, T_{\P_1})\rightarrow H^1(\tilde{Q},T_{\tilde{Q}})\rightarrow 0$$
$$0\rightarrow H^0(\P^1\times\P^1,T_{\P^1\times\P^1})\to W\rightarrow H^1(\tilde{Q},T_{\tilde{Q}})\rightarrow 0$$
signifying that every first-order infinitesimal deformation of $\P_1$, and therefore of $(\P_1, B_1\cup B_2)$, arises from
a first-order infinitesimal deformation of the points $Q_1, \ldots, Q_6$; $P_1,\ldots,P_4$ in $\P^1\times\P^1$ (i.e., an element of $W$) or a first-order deformation of $Q_3, \ldots, Q_6, P_1,\ldots,P_4$ in $\tilde{Q}$, or both.

%\bigskip $\clubsuit$ EXPAND: This is well-known (deformations of blow-ups are blow-ups,
%(-1)-curves deform, etc.), It basically follows from the exact sequence $0\to T_S(-\log E)\to T_S\to N_{E/S}\to0$
%where $E$ is a $(-1)$-curve on $S$: the sheaf on the right is acyclic. Maybe we can add a reference?
%For example, to Wahl's papers on equisingular deformations. \bigskip

%Viewing $H^1(\P_1,T_{\P_1}-\log(B_1+B_2))$ as the space of infinitesimal first-order deformations of $\P_1$ which preserve $B_1$ and $B_2$, we see that the map~\eqref{issurjective} factors through $H^1(\P_1, T_{\P_1})$, the space of infinitesimal deformations of $\P_1$. As $\P_1$ is the blowup of $\P^1\times \P^1$ in ten points, four of which are infinitely near, the space $H^1(\P_1, T_{\P_1})$ is isomorphic to $H^1(\P^1\times \P^1, ...)$, which we recognize as the space of infinitesimal first-order deformations of the points $Q_1, Q_2, R_1, \ldots R_4$, together with the tangent directions of $R_1, \ldots R_4$ on $\P^1\times \P^1$. In view of this, we see that $H^1(\P_1,\N_{\Delta/\P_1})$ corresponds to deformations of the points $Q_1, Q_2, R_1, \ldots R_4$, and the tangent directions of $R_1, \ldots R_4$, which fail to preserve $\Delta$. Thus, the map
%$$H^1(\P_1, T_{\P_1}-\log(B_1+B_2))\rightarrow H^1(\P_1,\N_{\Delta/\P_1})$$
%is surjective as long as given any deformation of the points $Q_1, Q_2, R_1, \ldots R_4$, and the tangent directions of $R_1, \ldots R_4$, there exists a deformation of $B_1$, $B_2$....

We note that \eqref{qqqyyyy}, and even $V_1\rightarrow H^1(\P_1,\N_{\Delta/\P_1})$, is surjective, i.e.~
each obstruction in $H^1(\P_1,\N_{\Delta/\P_1})$
arises from a first-order infinitesimal deformation of
$Q_1, \ldots, Q_6$ and the tangent directions of $Q_3, \ldots, Q_6$ in $\P^1\times\P^1$ that fails to induce a first-order embedded deformation of  $\Delta$.

\begin{lemma}\label{basis}
The space $H^1(\P_1,\N_{\Delta/\P_1})$ has dimension $7$ and has
the following distinguished basis. Each basis element comes from a first order deformation of $\P_1$ which fixes $Q_1$, $Q_2$, $Q_3$
(this takes care of infinitesimal automorphisms of $\P^1\times\P^1$) and
either
\begin{itemize}
\item $I_k$ for $k=1, 2, 3$ leaves the tangent direction at $Q_3$ fixed, i.e.~parallel to $\Delta$, and moves $Q_{k+3}\in \{Q_4, Q_5, Q_6\}$ off $\Delta$ while keeping the remaining points and their tangent directions
fixed, i.e.~parallel to $\Delta$,
or
\item $I_k$ for $k=4, 5, 6, 7$ fixes $Q_{k-1}\in\{Q_3, Q_4, Q_5, Q_6\}$ and changes the tangent direction at $Q_{k-1}$, moving the remaining points of $Q_4, Q_5, Q_6$ along $\Delta$ and keeping the tangent directions at these remaining points fixed, i.e.~parallel to $\Delta$.
\end{itemize}
\end{lemma}
\begin{proof} Simple calculation.
\end{proof}

To show that the map~\eqref{issurjective?} is surjective, it suffices to show that for each deformation type listed, there exists an equisingular deformation of $B_1\cup B_2$ in $\P^1\times\P^1$ which passes through the points to which $Q_1, \ldots, Q_6$ deform and which has the desired tangent direction at each point.

To begin, let us choose bi-homogeneous coordinates $((\a: \a'),( \b: \b'))$ on $Q=\P^1\times\P^1$ so that $ \a=\frac{x}{y}=-\frac{t}{z}$ and $ \b=\frac{x}{t}=-\frac{y}{z}$.
Let  $g_1$ and $g_2$ be the equations (bihomogeneous of degree $(3,3)$) of $B_1$ and $B_2$, respectively.
Referring to \eqref{b1eqqqq} and \eqref{b2eqqqq}, we have
$$g_1=
-\a \a'^2 \b^3
+3 \a^2 \a' \b^2 \b'
-3 \b^2 \b' \a'^3
-3 \a \a'^2 \b \b'^2
+3 \a^3 \b \b'^2
+ \a^2 \a' \b'^3 $$
$$g_2=
- \b \b'^2 \a'^3
-3 \a \a'^2 \b^2 \b'
+3 \a \a'^2 \b'^3
-3 \b \b'^2 \a^2 \a'
+3 \b^3 \a^2 \a'
- \b^2 \b' \a^3.
$$

Global first order deformations $\tilde{B}_1$ and $\tilde{B}_2$ of $B_1$ and $B_2$ are given by equations $$g_1+\epsilon\bar{g_1}=g_1+\epsilon\sum_{0\le i, j\le 3}a_{ij} \a^i \a'^{3-i} \b^j \b'^{3-j}$$
 and
$$g_2+\epsilon\bar{g_2}=g_2+\epsilon\sum_{0\le i, j\le 3}b_{ij} \a^i \a'^{3-i} \b^j \b'^{3-j},$$
respectively. In order to describe equisingular first-order deformations of $B_1\cup B_2$, we move the singularities of $B_1$ and $B_2$ at $P_1, \ldots, P_4$ to the points $(\epsilon c_1, \epsilon d_1), \ldots, (\epsilon c_4, \epsilon d_4)$, given in local coordinates on $U_1, \ldots, U_4 \subset Q$ respectively, where
$$U_1=\{\a=\b=1\}, \quad U_2=\{\a'=\b=1\}, $$
$$U_3=\{\a=\b'=1\}, \quad  U_4=\{\a'=\b'=1\}.$$
To simplify calculations, we change coordinates on $U_1, \ldots, U_4$, so that the points $(\epsilon c_1, \epsilon d_1), \ldots, (\epsilon c_4, \epsilon d_4)$ are at the origin.

Letting $g_{ij}+\epsilon \bar{g}_{ij}$ be the degree $j$ part of the equation $g_i+\epsilon \bar{g}_i$ with respect to the new coordinates, we have the following conditions. These ensure that $B_1\cup B_2$ maintains the singularities, with possibly different tangent cones, at the points to which $P_1, \ldots, P_4$ deform. For simplicity we use the same notation for $P_1, \ldots, P_4$ and the points to which they deform.

\begin{enumerate}
\item $g_{10}+\epsilon\bar{g}_{10}=0$ on each $U_i$. This forces $\tilde{B}_1$ to pass through $P_1, \ldots, P_4$.
\item $g_{11}+\epsilon\bar{g}_{11}=0$ on $U_1$, $U_4$. This forces $\tilde{B}_1$ to be singular at $P_1$ and $P_4$.
\item $g_{12}+\epsilon\bar{g}_{12}=(m+m_1\epsilon)(g_{21}+\epsilon \bar{g}_{21})^2$, for some constants $m, m_1$, on $U_1$, $U_4$  (where $m, m_1$ may differ on $U_1$, $U_4$). This forces the tangent cones of $\tilde{B}_1$ at $P_1$ and $P_4$ to be the same as those of $\tilde{B}_2$ at $P_1$ and $P_4$.
\item $g_{13}+\epsilon\bar{g}_{13}=(g_{21}+\epsilon\bar{g}_{21})(h+\epsilon h_1)$, where $h$ and $h_1$ are quadratic forms. By Lemma~\ref{tacnode}, this forces $\tilde{B}_1$ to have tacnodes at the points $P_1$ and $P_4$.

\item $g_{20}+\epsilon\bar{g}_{20}=0$ on each $U_i$.  This forces and $\tilde{B}_2$ to pass through $P_1, \ldots, P_4$.
\item $g_{21}+\epsilon\bar{g}_{21}=0$ on $U_2$, $U_3$.  This forces $\tilde{B}_2$ to be singular at $P_2$ and $P_3$.
\item $g_{22}+\epsilon\bar{g}_{22}=(n+n_1\epsilon)(g_{11}+\epsilon\bar{g}_{11})^2$, for some constants $n, n_1$, on $U_2$, $U_3$ (where $n, n_1$ may differ on $U_2$, $U_3$). This forces the tangent cones of $\tilde{B}_2$ at $P_2$ and $P_3$ to be the same as those of $\tilde{B}_1$ at $P_2$ and $P_3$.
\item $g_{23}+\epsilon\bar{g}_{23}=(g_{11}+\epsilon\bar{g}_{11})(h+\epsilon h_1)$, where $h$ and $h_1$ are quadratic forms. By Lemma~\ref{tacnode}, this forces $\tilde{B}_2$ to have tacnodes at the points $P_2$ and $P_3$.
\end{enumerate}

Returning to original coordinates, and after simple algebraic manipulations,
this gives the following system of $28$ linear equations in $c_i, d_i, a_{ij}, b_{ij}$
(four blocks for four charts):

\begin{equations}\label{myequations}

$$a_{33}=0$$%degree 0 part of g1bar
$$b_{33}=d_1-3c_1$$%degree 0 part of g2bar
$$a_{32}=-3c_1-6d_1$$%degree 1 part of g1bar, coefficient of beta'
$$a_{23}=2c_1-3d_1$$%degree 1 part of g1bar, coefficient of alpha'
$$a_{22}=a_{31}+b_{23}+3b_{32}$$%degree 2 part of g1bar, coeff of alpha'beta'
$$a_{13}=2a_{31}-2b_{32}+4b_{23}$$%degree 2 part of g1bar, coeff of (alpha')^2
$$2c_1-d_1+3a_{12}+a_{03}+2a_{21}-a_{30}=0$$ %degree 3 part of g1bar

$$a_{30}=-c_2-3d_2$$ %degree 0 part of g1bar
$$b_{30}=0$$%degree 0 part of g2bar
$$b_{31}=3c_2+2d_2$$%degree 1 part of g2bar, coeff of beta
$$b_{20}=3d_2+c_2$$%degree 1 part of g2bar, coeff of alpha'
$$b_{32}=2b_{10}+4a_{31}+2a_{20}$$%degree 2 part of g2bar, coeff of beta^2
$$b_{21}=6b_{10}+6a_{31}+3a_{20}$$%degree 2 part of g2bar, coeff of beta*alpha'
$$5c_2+3d_2+5b_{00}+3b_{11}+6b_{22}+5b_{33}=0$$%degree 3 part of g2bar

$$a_{03}=c_3+3d_3$$%degree 0 part of g1bar
$$b_{03}=0$$% degree 0 part of g2bar
$$b_{02}=2d_3+3c_3$$% degree 1 part of g2bar coefficient of beta'
$$b_{13}=3d_3+c_3$$% degree 1 part of g2bar coefficient of alpha
$$b_{01}=5a_{13}+2b_{23}+3a_{02}$$%degree 2 part of g2bar, coeff of beta'^2
$$b_{12}=4a_{13}+6b_{23}+a_{02}$$%degree 2 part of g2bar, coeff of beta'alpha
$$c_3+2d_3-3b_{11}+b_{33}+2b_{22}+b_{00}=0$$%degree 3 part of g2bar

$$a_{00}=0$$ %degree 0 part of g1bar
$$b_{00}=d_4-3c_4$$ %degree 0 part of g2bar
$$a_{01}=3c_4+6d_4$$ %degree 1 part of g1bar beta term
$$a_{10}=3d_4-2c_4$$ %degree 1 part of g1bar alpha term
$$a_{20}=2a_{02}+2b_{01}+3b_{10}$$ %degree 2 part of g1bar alpha*beta term
$$a_{11}=a_{02}+4b_{01}+6b_{10}$$ %degree 2 part of g1bar alpha^2 term
$$4c_4+5d_4+2a_{03}+3a_{12}+a_{21}+5a_{30}=0$$%degree 3 part of g1bar
\end{equations}

Next, we determine all additional conditions on $a_{ij}$, $b_{ij}$, $c_i$, $d_i$ which ensure that $\tilde{B}_1$ and $\tilde{B}_2$ pass through the points to which $Q_1, \ldots, Q_6$ deform, with the desired multiplicities at each point. To do so, we look in the chart $U_4$. Here, the equation of $\Delta$ is
$$\a(1+\b)+\b-1=0.$$
Solving for $\a$ gives
$$\a=\frac{1-\b}{1+\b}.$$
Thus, the points at which $\Delta$ intersects $B_1$ and $B_2$ are the roots of the following polynomials:
$$(\b^2+1)(\b^2+4\b+6)^2$$
and
$$(\b^2+1)(\b^2+6\b+6)^2.$$
This gives the six points at which $B_1$ and $B_2$ intersect $\Delta$:
$$Q_1=(-i, i), \quad  Q_2=(i, -i),$$
$$Q_3=(3-5i, -2+4i),\quad Q_4=(3+5i, -2-4i),$$
$$Q_5=(-5+4i, -3+5i), \quad Q_6=(-5-4i, -3-5i),$$
where $i^2+1=0 \textrm{ mod }7$.

The intersections of $\bar{g}_1=0$ and $\bar{g}_2=0$ with $\Delta$ are given by the zeros of the following polynomials:
$$\hat{g}_1=
(1+\b)^3(a_{00}+a_{01}\b+a_{02}\b^2+a_{03}\b^3)$$
$$+(1+\b)^2(1-\b)(a_{10}+a_{11}\b+a_{12}\b^2+a_{13}\b^3)$$
$$+(1+\b)(1-\b)^2(a_{20}+a_{21}\b+a_{22}\b^2+a_{23}\b^3)$$
$$+(1-\b)^3(a_{30}+a_{31}\b+a_{32}\b^2+a_{33}\b^3)$$

$$\hat{g}_2=(1+\b)^3(b_{00}+b_{01}\b+b_{02}\b^2+b_{03}\b^3)$$
$$+(1+\b)^2(1-\b)(b_{10}+b_{11}\b+b_{12}\b^2+b_{13}\b^3)$$
$$+(1+b)(1-\b)^2(b_{20}+b_{21}\b+b_{22}\b^2+b_{23}\b^3)$$
$$+(1-\b)^3(b_{30}+b_{31}\b+b_{32}\b^2+b_{33}\b^3)$$

Using these equations, we obtain $8$ additional linear equations in $a_{ij}$, $b_{ij}$, $c_i$, $d_i$. These ensure that $\tilde{B}_1$ and $\tilde{B}_2$ pass through $Q_1, Q_2$, that $\tilde{B}_1$ passes through $Q_3, Q_4$, and that $\tilde{B}_2$ passes through $Q_5, Q_6$. Note that each restriction arises from setting $\beta$ equal to $i$, $-i$,  $-2+4i$, $-2-4i$, $-3+5i$, or $-3-5i$ in the appropriate equation.

\begin{itemize}
\item[(B1Q1)]
$$ %B1Q1
(3c_1 - 3c_4 + 3d_1 + d_4 - 2a_{20} - 2a_{21} - a_{31} + 3a_{02} - a_{12} + 3a_{03} - 2b_{10} + 3b_{32} + 2b_{23})i $$
$$+ 3c_1 - 3c_4 + 3d_1 + d_4 + 2a_{20} - 2a_{21} + a_{31} - 3a_{02} - a_{12} + 3a_{03} + 2b_{10} - 3b_{32} - 2b_{23}=0$$
\item[(B1Q2)]
$$ %B1Q2
(- 3c_1 + 3c_4 - 3d_1 - d_4 + 2a_{20} + 2a_{21} + a_{31} - 3a_{02} + a_{12} - 3a_{03} + 2b_{10} - 3b_{32} - 2b_{23})i$$
$$ + 3c_1 - 3c_4 + 3d_1 + d_4 + 2a_{20} - 2a_{21} + a_{31} - 3a_{02} - a_{12} + 3a_{03} + 2b_{10} - 3b_{32} - 2b_{23}=0$$
\item[(B2Q1)]
$$ %B2Q1=
(- 3c_2 - c_3 - 3d_3 - a_{20} - 3a_{31} -a_{02} - 2b_{11} + b_{22} + 3b_{32} + b_{23})i$$
$$ + 3c_2 + c_3 + 3d_3 - a_{20} - 3a_{31} - a_{02} + 2b_{11} - b_{22} + 3b_{32} + b_{23}=0$$
\item[(B2Q2)]
$$%B2Q2=
(3c_2 + c_3 + 3d_3 + a_{20} + 3a_{31} + a_{02} + 2b_{11} - b_{22} - 3b_{32} - b_{23})i$$
$$ + 3c_2 + c_3 + 3d_3 - a_{20} - 3a_{31} - a_{02} + 2b_{11} - b_{22} + 3b_{32} + b_{23}=0$$
\item[(B1Q3)]
$$%B1Q3=
(- c_4 + 2d_1 - 2d_4 + 3 a_{20} + 3 a_{21} -  a_{31} + 3 a_{12} + 2 b_{10} - 2 b_{32} - 3 b_{23})i$$
$$ + 3c_1 - c_4 - 3d_1 - 2d_4 - 2a_{20} + 3a_{21} - 3a_{31} + a_{02} + a_{12} - a_{03} - 3b_{10} - 3b_{32} + b_{23}=0$$
\item[(B1Q4)]
$$%B1Q4=
(c_4 - 2 d_1 + 2 d_4 - 3 a_{20} - 3 a_{21} +  a_{31} - 3 a_{12} - 2 b_{10} + 2 b_{32} + 3 b_{23})i$$
$$ + 3c_1 - c_4 - 3d_1 - 2d_4 - 2a_{20} + 3a_{21} - 3a_{31} + a_{02} + a_{12} - a_{03} - 3b_{10} - 3b_{32} + b_{23}=0$$
\item[(B2Q5)]
$$%B2Q5=
(c_1 + 3 c_2 + 2 c_3 + 2 d_1 -  d_3 +  a_{20} - 3 a_{31} + 2 a_{02} -  b_{10} + 3b_{11} -  b_{22} +  b_{23})i$$
$$ + 2c_1 - 2c_2 + 2c_3 - 3d_1 + 3d_2 - d_3 - 2a_{20} - a_{31} + 3a_{02} + b_{10} + 3b_{11} - 3b_{22} + b_{32} - 3b_{23}=0$$
\item[(B2Q6)]
$$%%B2Q6=
(-  c_1 - 3 c_2 - 2 c_3 - 2 d_1 +  d_3 -  a_{20} + 3 a_{31} - 2 a_{02} +  b_{10} - 3b_{11} +  b_{22} -  b_{23})i$$
$$ + 2c_1 - 2c_2 + 2c_3 - 3d_1 + 3d_2 - d_3 - 2a_{20} - a_{31} + 3a_{02} + b_{10} + 3b_{11} - 3b_{22} + b_{32} - 3b_{23}=0$$
\end{itemize}

Taking the derivatives of $\hat{g}_1$ and $\hat{g}_2$ with respect to $\b$ and setting $\b$ equal to $-2+4i$, $-2-4i$, $-3+5i$, or $-3-5i$ as appropriate gives the final four linear equations in $a_{ij}, b_{ij}, c_i, d_i$. These ensure that $\tilde{B}_1$ and $\tilde{B}_2$ are tangent to $\Delta$ at $Q_3, Q_4$ and $Q_5, Q_6$, respectively.

\begin{itemize}
\item[(dB1Q3)]
$$%dB1Q3=
(3c_1 - 2c_4 - 3d_1 + 3d_4 + a_{21} + 2a_{31} + 2a_{02} + a_{03} - b_{10} - 2b_{32})i$$
$$ - c_1 + c_4 - 2d_1 + 2d_4 + 2a_{20} + a_{21} + 3a_{02} + a_{03} - 3b_{10} - 2b_{32} - b_{23}=0$$
\item[(dB1Q4)]
$$%dB1Q4=
(- 3c_1 + 2c_4 + 3d_1 - 3d_4 - a_{21} - 2a_{31} - 2a_{02} - a_{03} + b_{10} + 2b_{32})i$$
$$ - c_1 + c_4 - 2d_1 + 2d_4 + 2a_{20} + a_{21} + 3a_{02} + a_{03} - 3b_{10} - 2b_{32} - b_{23}=0$$
\item[(dB2Q5)]
$$%dB2Q5=
(c_1 - 3c_2 - c_3 + 2d_1 + d_2 - 3d_3 + 3a_{20} + 2a_{31} + 3a_{02} + 2b_{10} - 2b_{11} - 3b_{22} - 2b_{23})i$$
$$ - c_2 - 3c_3 - 3d_2 - 2d_3 - 3a_{31} - 3a_{02} - b_{10} - b_{22} - 3b_{32}=0$$
\item[(dB2Q6)]
$$%dB2Q6=
(-c_1 + 3c_2 + c_3 - 2d_1 - d_2 + 3d_3 - 3a_{20} - 2a_{31} - 3a_{02} - 2b_{10} + 2b_{11} + 3b_{22} + 2b_{23})i$$
$$- c_2 - 3c_3 - 3d_2 - 2d_3 - 3a_{31} - 3a_{02} - b_{10} - b_{22} - 3b_{32}=0$$
\end{itemize}

Consider a basis element in $H^1(\P_1, \N_{\Delta/\P_1})$ corresponding via Lemma~\ref{basis} to some deformation of the points $P_1, \ldots,P_4, Q_1,\ldots, Q_6$ together with the tangent directions of $P_1, \ldots, P_4, Q_4, \ldots, Q_6$ in $\P^1\times \P^1$.  There are two cases, as in Lemma~\ref{basis}.

Consider for example the basis element $I_1$.  The existence of an equisingular deformation of $(\P_1, B_1\cup B_2)$ mapping to $I_1$ is equivalent to the existence of $a_{ij}, b_{ij}, c_i, d_i$ which satisfy Equations~\ref{myequations}, as well as B1Q1, B1Q2, B2Q1, B2Q2, B1Q3, dB1Q3, B1Q4-1, B2Q5, B2Q6. Here, we use Lemma~\ref{defs}.

Next we consider the basis element $I_4$. The existence of an equisingular deformation of $(\P_1, B_1\cup B_2)$ mapping to $I_4$ is equivalent to the existence of $a_{ij}, b_{ij}, c_i, d_i$ which satisfy Equations~\ref{myequations}, as well as B1Q1, B1Q2, B2Q1, B2Q2, B1Q3, dB1Q3, dB1Q3-1, B2Q4, B2Q5, B2Q6. Here, we use Lemma~\ref{defs}.

Thus, each of the seven basis elements corresponds to finding a nontrivial solution of a large system of linear equations. As working with such large matrices is unwieldy, we use Macaulay2 to check this (see the code included in the Appendix). In each case, we find that solutions indeed form either a 3- or 4-dimensional vector space, depending on the basis element, completing the proof.
\end{proof}

\begin{lemma}\label{tacnode}
The singularity $(h_1^2+h_1h_2+h.o.t.=0)\subset\A^2$, where $h_i$ is a form of degree $i$, is a tacnode (or a degeneration of a tacnode).
\end{lemma}
\begin{proof} Completing the square, the equation becomes $((h_1+\frac{1}{2}h_2)^2+h.o.t=0)\subset \A^2$. Letting $h=h_1+\frac{1}{2}h_2$, the singularity becomes
$(h^2+h.o.t.=0)\subset \A^2$. As there are no terms of degree $3$, this is a tacnode.
\end{proof}

\begin{lemma}\label{defs}
Let $B=(g=0)$ be the germ of a smooth curve in $\A^2$ which is simply tangent to the $x$-axis at the origin, and let $\tilde{B}=(g+\epsilon \bar{g}=0)$ be its first order infinitesimal embedded deformation. Then $\tilde{B}$ is tangent to the $x$-axis if and only if $\bar{g}(0,0)=0$.
\end{lemma}
\begin{proof}  Suppose $\bar{g}(0,0)=0$. We have to show that there exists $x_0$ with
$$g(\epsilon x_0, 0)+\epsilon \bar{g}(\epsilon x_0, 0)=0$$
and that
$$\frac{d}{dx}\left(g(\epsilon x_0, 0)+\epsilon \bar{g}(\epsilon x_0, 0)\right)=0.$$
Taking the Taylor expansion of these with respect to $\epsilon$, the first of these obviously holds. The second holds for
$$x_0=\frac{-\bar{g}_1'(0,0)}{g_1''(0,0)}.$$
\end{proof}

%----------------------------------------------------------------------------------------------------
\section{The Boyd surface is a Dolgachev surface}\label{boydgeo2}

\begin{lemma}
Blowing-down $\bar N_1$ and $\bar N_2$ on $\P$ gives a Halphen surface of index~$3$ \cite[Ch.V Sect.6]{CD}
with a multiple fiber $A_1+\ldots+A_4+ \bar G_1+\ldots\bar G_4$ of type $I_8$.
\end{lemma}

\begin{proof}
By \eqref{sdghfghfgh}, $\P$ has a fibration $\P \to \P_k^1$ with connected fibers such that the general fiber is smooth of genus $1$; see \cite[Sect.7]{B01}. Moreover, the $I_8$ fiber $\sum_{i=1}^4 A_i + \sum_{i=1}^4 \bar G_i$ has multiplicity $3$. Thus this elliptic fibration is a Halphen surface of index $3$ (after one blows-down $\bar N_1$ and $\bar N_2$); see \cite[Ch.V Thm.5.6.1]{CD}.
\end{proof}

\begin{lemma}\label{DDDD7}
The Boyd surface $Y$ is a Dolgachev surface in characteristic $7$. The elliptic fibration $Y \to \P_k^1$ has four singular fibers: one $I_4$ with multiplicity $3$, one $I_4$ with multiplicity $2$, and two reduced $I_2$.
\end{lemma}

\begin{proof}
We denote by $\alpha$ the composition $W \to \P \to \P_k^1$. Since this is a projective morphism, we have a Stein factorization for $\alpha$, i.e.~maps $\beta \colon W \to C$ with connected fibers and $\gamma \colon C \to \P_k^1$ a finite morphism such that $\alpha=\gamma \circ \beta$. Notice that the multiplicity of the fiber $B_1+B_2+ \bar N_1 + \bar N_2$ of $\P \to \P_k^1$ is $1$, and so $\gamma \colon C \to \P_k^1$ is a finite separable morphism. Notice also that the fibers $B_1+B_2+ \bar N_1 + \bar N_2$ and $I_8$ in $\P \to \P_k^1$ pull back to connected fibers of $\alpha$ with multiplicities $2$ and $3$ respectively. Since these multiplicities are coprime, we must have that the degree of $\gamma$ is one, and so $\gamma$ is an isomorphism. In this way $\alpha$ has connected fibers. In addition, since it has two multiple fibers, the Kodaira dimension of $Y$ is nonnegative \cite{CD}.

The double cover $W \to \P$ induces a connected \'etale cover between the non multiple fibers of $\alpha$. Notice that $\P \to \P_k^1$ can only have irreducible singular fibers apart from $B_1+B_2+ \bar N_1 + \bar N_2$ and $I_8$, because the Picard number of $\P$ is $12$. Therefore we can have either two $I_1$ or one $II$ as extra singular fibers. But a fiber of type $II$ is \'etale simply connected, and so it does not have a connected \'etale cover of degree $2$. Thus, $\P \to \P_k^1$ has precisely two extra $I_1$ singular fibers, and their pre-images under $W \to \P$ give two $I_2$ reduced fibers for $\alpha$. This elliptic fibration induces a relatively minimal elliptic fibration $Y \to \P_k^1$, after we blow-down the curves $G_1,\ldots,G_4$.

Using well-known facts on double covers, one can easily verify that $K_Y^2=0$, $\chi(\O_Y)=1$, and
\begin{equation}\label{pg0}
p_g(Y)=h^2(-L)=h^0(K_{\P}+L)=0,
\end{equation}
where $$L=3 \sigma^*(\Delta)- \sum_{i=1}^4 \bar E_i - 2 \sum_{i=1}^4 \bar G_i - \bar N_1 - \bar N_2$$ is the line bundle defining the double cover $\pi'$. Thus $q(Y)=0$.
\end{proof}

The previous Lemma shows the canonical class of $Y$ has the form
\begin{equation}\label{KBoyd}
K_Y \sim -F + \Gamma_2 + 2 \Gamma_3 \equiv 1/6 F,
\end{equation}
where $F$ is a general fiber, $\Gamma_2$ is the $I_4$ with multiplicity $2$, and $\Gamma_3$ is the $I_4$ with multiplicity $3$.

\begin{lemma}\label{nefness}
$K_{S_0}$ is nef.
\end{lemma}

\begin{proof}
The Boyd surface $Y$ is the minimal resolution of the surface $S_0$, which has log terminal singularities. Therefore, it suffices to show that $K_Y$ is nef,
which follows from \eqref{KBoyd}.
\end{proof}

%-----------------------------------------------------------------------------------------------------------------------------------------
\section{Some mixed characteristic deformation theory}\label{DefSec}

In this section we show that the Craighero--Gattazzo surface can be degenerated to a special complex
surface with a ${1\over 4}(1,1)$ singularity. Our argument is
based on the following simple fact.

\begin{lemma}\label{asgrhsf}
 Let $\cR$ be a DVR with residue field~$k$
 and quotient field $K$. Let $\bar K$ be the algebraic closure of $K$.
 Let $T$ be a smooth $\cR$-scheme.
 Let $o\in T$ be a $k$-point.
 Let~$\sigma_1,\sigma_2:\,\Spec\cR\to T$
 be two sections passing through  $o$.
 Then there exists an irreducible smooth $\bar K$-curve $C$ and a morphism
 $C\to T_{\bar K}$ such that its image contains $\sigma_1(\eta)$ and $\sigma_2(\eta)$, where
 $\eta\in\Spec\cR$ is the generic point.
\end{lemma}

\begin{remark}
For the proof we only need $\sigma_1$ to be a section; $\sigma_2$ can be
~a section $\Spec\cR'\to T_{\cR'}$ after a finite surjective base change $\Spec\cR'\to \Spec\cR$.
\end{remark}

\begin{proof}
We can substitute $T$ with an affine connected component $\Spec A$ of~$o$.
By \cite[p.56]{Mum}, it suffices to prove that $T_K$ is geometrically connected.
Since it is smooth over $\Spec K$ and has a $K$-point $\sigma_1(\eta)$,
it suffices to prove that it is connected.
Arguing by contradiction, suppose it is disconnected. Then $H^0(T_K,\O_{T_K})$ contains
a non-trivial idempotent $e$.
Let $\pi\in\cR$ be a uniformizer. Since $T$ is flat over $\Spec\cR$, $\pi$ is not a zero-divisor in $A$, and so~$e\in A[1/\pi]$. Let $n$ be the minimal non-negative integer such that $e$ can be written as $a/\pi^n$ with $a\in A$.
Then $a^2=\pi^na$. Since $T$ is  smooth over $\Spec\cR$, its special fiber is reduced.
It follows that $n=0$ because otherwise $a^2=0\mod(\pi)$ and therefore $a=0\mod(\pi)$, which implies that $n$ is not minimal.
So $e\in A$, which contradicts connectedness of $T$.
\end{proof}

\begin{lemma}\label{afbadfbadfb}
Let $\cR$ be a complete DVR with residue field~$k$
 and quotient field~$K$.
 Let $\bar K$ be the algebraic closure of $K$.
 Let $F$ be a limit preserving contravariant functor
 from the category of $\cR$-schemes to the category of sets.

 Fix $\zeta_0\in F(\Spec k)$.
 Let $F_{\zeta_0}$ be its ``deformation functor'', i.e.~
 a functor
 from the category of pointed $\cR$-schemes $(X,x_0)$,
 where $x_0$ is a closed point with residue field $k$, to sets. Specifically, $F_{\zeta_0}(X,x_0)=\{\xi\in F(X),|\,F(i)\xi=\zeta_0\}$,
 where $i:\,\Spec k=\Spec k(x_0)\hookrightarrow X$ is the inclusion.

 Suppose the restriction of $F_{\zeta_0}$ to the category of spectra of local artinian $\cR$-algebras with residue field $k$
 is smooth and satisfies Schlessinger's conditions \cite{Sch}. Suppose also that the natural map
 \begin{equation}\label{ajhsbvakshjf}
 F_{\zeta_0}(\Spec A)\to\lim_{\longleftarrow}F_{\zeta_0}(\Spec A/\m^n)
 \end{equation}
 is bijective for every complete local Noetherian $\cR$-algebra $(A,\m)$ with residue field~$k$.

 Let $\Sigma_1,\Sigma_2\in F_{\zeta_0}(\Spec\cR)$ and let $\bar\Sigma_1,\bar\Sigma_2\in F(\Spec\bar K)$
be their pull-backs to $\Spec\bar K$. Then there exists an irreducible smooth $\bar K$-curve $C$, $\bar K$-points $y_1,y_2\in C$,
 and an element $\Sigma\in F(C)$ which restricts to $\bar\Sigma_1$ and $\bar\Sigma_2$ at $y_1$ and $y_2$, respectively.
\end{lemma}

\begin{proof}
By \cite{Sch}, $F_{\zeta_0}$ admits a hull, and by \eqref{ajhsbvakshjf} we can assume that
the hull is induced by an element $\bar\zeta\in F_{\zeta_0}(\Spec \cH)$, where
$(\cH,\m)$  is a complete local Noetherian $\cR$-algebra with residue field $k$.
By Artin's algebraization theorem \cite[Th.~1.6]{Ar}, there exists an $\cR$-scheme of finite type $T$,
a closed $k$-point $o\in T$, an element $\zeta\in F_{\zeta_0}(T,o)$,
and an isomorphism $\sigma:\,\hat\O_{T,o}\to\cH$ such that $F(\sigma)\zeta$ and $\bar\zeta$ agree on $\cH/\m^n$ for all $n\ge1$.
By \eqref{ajhsbvakshjf}, in fact $F(\sigma)\zeta=\bar\zeta$.

Since $F_{\zeta_0}$ is smooth,
$T\to\Spec\cR$ is formally smooth at $o$, and therefore we can assume that $T$ is a smooth $\cR$-scheme after shrinking it if necessary.

Since $\cR$ is complete, we can find sections $\sigma_1,\sigma_2:\,\Spec\cR\to T$
such that $F(\sigma_i)(\zeta)$ and $\Sigma_i$ agree on $\cR/\n^n$ for  any~$n\ge1$, where $\n\subset\cR$
is the maximal ideal. By \eqref{ajhsbvakshjf}, $F(\sigma_i)(\zeta)=\Sigma_i$.
It~remains to apply Lemma~\ref{asgrhsf}.
\end{proof}

In our application $F$ will be a functor of  $\Q$-Gorenstein deformations,
as worked out in \cite{Ha} in characteristic zero and
\cite{AH} in general.
For simplicity, we allow only Cohen--Macaulay surfaces.
Following \cite{AH}, let $\cK^\omega$ be the
category  of Koll\'ar families fibered in groupoids over the category of schemes.
An object of  $\cK^\omega$ over a scheme $B$ is a triple $(f:\,X\to B, F, \phi)$,
where $f$ is a proper flat family of connected reduced Cohen--Macaulay surfaces, $F$~is a coherent sheaf,
and $\phi:\,F\to\omega_{X/B}$ is
an isomorphism. Moreover, we assume that the formation of every reflexive power $F^{[n]}$
commutes with arbitrary base change (we call this the Koll\'ar condition) and that for every geometric point $s$ of $B$
there exists a positive integer $N_s$ such that $F^{[N_s]}|_{X_s}$ is invertible and ample.
See \cite{AH} for the description of morphisms in $\cK^\omega$ and for the proof that it  is an algebraic stack.
The functor $\Def^{\Q G}$ of $\Q$-Gorenstein deformations is the associated set-valued functor of
isomorphism classes of Koll\'ar families.

\begin{theorem}\label{abstract}
Let $\cR$ be a complete DVR with algebraically closed  residue field~$k$ and quotient field $K$.
Let $\bar K$ be the algebraic closure of $K$.
Let $\cX_1$ and $\cX_2$ be two $\Q$-Gorenstein families over $\Spec\cR$.
Suppose their special fibers are both isomorphic to a $k$-surface $X$.
Let $\cK^\omega_\cR$ be the restriction of $\cK^\omega$ to the category of $\cR$-schemes.
Suppose it is $\cR$-smooth at $X\to\Spec k$.
Then there exists an irreducible smooth $\bar K$-curve $C$, $\bar K$-points $y_1,y_2\in C$,
 and a $\Q$-Gorenstein family over $C$ with
 fibers at $y_1$ and $y_2$ isomorphic to
 $(\cX_1)_\oK$ and $(\cX_2)_\oK$, respectively.
 \end{theorem}

\begin{proof}
Since $\cK^\omega_{\cR}$ is an algebraic $\cR$-stack, its associated set-valued functor $\Def^{\Q G}_{\cR}$
satisfies the conditions of
Lemma~\ref{afbadfbadfb} by Artin's criterion \cite{Ar1}.
\end{proof}

In our situation, $\cX_1$ will be a degeneration of the Craighero--Gattazzo surface
to the contraction $S_0$ of the Boyd surface $Y$. To construct the second family, we will need the following well-known fact.

\begin{lemma}\label{ajhsdbafhjsfb}
 Let $k$ be an algebraically closed field, let $\cR$ be a complete DVR with residue field $k$,
 let $Y$ be a smooth projective surface over $k$ and let $C_1,\ldots,C_r\subset Y$ be smooth curves intersecting transversally.
 Suppose
 $$H^2(Y, T_Y(-\log(C_1+\ldots +C_r)))=H^2(Y, \O_Y)=0.$$
Then there exists a smooth projective family of surfaces $\Y\to\Spec\cR$
with closed subschemes $\cC_1, \ldots, \cC_r\subset\Y$ smooth and proper over $\Spec\cR$
such that the special fiber is $(Y; C_1,\ldots,C_r)$.
\end{lemma}

\begin{proof} This is well-known but we sketch a proof for completeness.
Let $\m\subset\cR$ be the maximal ideal and let $\cR_n=\cR/\m^{n+1}$ for each $n=0,1,\ldots$
We first lift $(Y; C_1,\ldots,C_r)$ to a scheme and a collection of subschemes flat over $\Spec\cR_n$ for each $n$ by induction on $n$.
So assume we already have a lift $(Y^n; C^n_1,\ldots,C^n_r)$ to $\Spec\cR_n$.
We have an exact sequence
\begin{equation}\label{kb,hjkgk,hjbvk,}
 0\to T_Y(-\log(C_1+\ldots +C_r))\to T_Y\to i_{1*}N_{C_1/Y}\oplus\ldots\oplus i_{r*}N_{C_r/Y}\to 0
\end{equation}
of sheaves on $Y$, where $i_{j}:\,C_j\to Y$ denotes the embedding for each~$j$.
Since $H^2(Y, T_Y(-\log(C_1+\ldots +C_r)))=0$, we have
$H^2(Y, T_Y)=0$ as well. Therefore we can lift $Y^n$ to a scheme $Y^{n+1}$ flat (and then automatically
smooth and proper) over $\Spec \cR_{n+1}$. Moreover, all possible lifts form an affine space
with underlying vector space $H^1(Y, T_Y)$.
Since
$$H^1(Y, T_Y)\to H^1(C_1,N_{C_1/Y})\oplus\ldots\oplus H^1(C_r, N_{C_r/Y})$$
is surjective by $H^2(Y, T_Y(-\log(C_1+\ldots +C_r)))=0$,
we can choose a lift such that the corresponding class in $H^1(C_i, N_{C_i/Y})$ vanishes for each~$i$.
This class can be interpreted as an obstruction to lifting $C^n_i\subset Y^n$ to a subscheme $C_i^{n+1}\subset Y^{n+1}$
flat over $\Spec \cR_{n+1}$.
So we can lift all $C_i$'s to subschemes $C^{n+1}_i\subset Y^{n+1}$ flat
(and automatically smooth and proper) over $\Spec \cR_{n+1}$.
The projective limit $\hat\Y=\lim\limits_\leftarrow Y^n$ is a formal scheme smooth and proper over $\Spf\cR$.
The projective limits $\hat\cC_i=\lim\limits_\leftarrow C_i^n$ for $i=1,\ldots,n$ are closed formal subschemes
smooth and proper over $\Spf\cR$.

Since $H^2(Y, \O_Y)=0$, we can lift any ample invertible sheaf on $Y$
to an (automatically ample) invertible sheaf on $\hat\Y$. By Grothendieck's existence theorem \cite[5.4.5]{EGAIII1},
there exists a scheme $\Y$ projective and
flat (and then automatically smooth) over $\Spec\cR$ such that $\hat\Y$ is a completion of its special fiber.
By~\cite[5.1.8]{EGAIII1}, there exist closed subschemes $\cC_1, \ldots, \cC_r\subset\Y$ such that
$\hat\cC_1,\ldots,\hat\cC_r$ are completions of their special fibers. They are flat (and automatically smooth and proper)
over $\Spec\cR$.
\end{proof}

\begin{notation}
We revert to the notation of the previous sections;
in particular $\cR$ will denote the ring of Witt vectors of an algebraically closed field $k$ of characteristic~$7$.
We~denote by $Y$~the Boyd surface over $k$. The $(-4)$-curve
$\Delta_1$ and the $(-2)$-curve $N_1$ of $Y$ intersect transversally and in one point.
\end{notation}

\begin{lemma}\label{sdgsgs}
There exists a smooth projective family of surfaces $\Y\to\Spec\cR$
with closed subschemes $\cC, \N\subset\Y$ smooth and proper over $\Spec\cR$
such that their geometric fibers are transversal rational curves of self-intersection $-4$ and $-2$,
respectively. The special fiber is the Boyd surface $(Y, \Delta_1, N_1)$.
\end{lemma}

\begin{proof}
This follows from Theorem~\ref{julie}, \eqref{pg0}, preservation of intersection numbers,
and  Lemma~\ref{ajhsdbafhjsfb}.
\end{proof}

We need a few facts about the ${1\over4}(1,1)$ singularity.
Let $\mu_4$ be the $\Z$-group scheme $\Spec\Z[\iota]/(\iota^4-1)$ with
comultiplication $\iota\to\iota\otimes\iota$.
Let $$\bX=\Spec \Z[u,v]^{\mu_4}=\Spec\Z[u^4, u^3v,u^2v^2,uv^3,v^4],$$
where $\mu_4$ acts on $\A^2$ with weights $(\iota,\iota)$.
For any scheme~$S$, we say that $\bX_S\to S$ is the standard family of surfaces with ${1\over4}(1,1)$ singularity.
If $k$ is a field then $\bX_k$ is isomorphic to the cone over the rational normal curve in~$\P^4_{ k}$.

\begin{definition}
Let $S$ be a locally Noetherian scheme and let $\X\to S$ be a flat family of geometrically connected reduced
surfaces smooth outside of a section $\Sigma:\,S\to\X$.
We say that $\X\to S$ has a ${1\over4}(1,1)$ singularity along~$\Sigma$ if there exists
a (not necessarily cartesian) commutative diagram
$$\begin{CD}
\X' @>g>> \X\\
@VVV @VVV\\
S'@>f>> S\\
\end{CD}$$
of morphisms with commuting sections $\Sigma$ and $\Sigma':\,S'\to\X'$
such that $f$ is surjective \'etale, $g$ is \'etale, and $\X'$ is isomorphic to an \'etale neighborhood of the section in the standard family $\bX_{S'}$.
\end{definition}

\begin{lemma}\label{kjbkhbk}
Let $\X\to S$ be a flat family of geometrically connected reduced surfaces with a section $\Sigma:\,S\to\X$ over a locally Noetherian base scheme $S$ and
smooth outside of $\Sigma$. Then
$\X$ has ${1\over4}(1,1)$ singularity along $\Sigma$ if and only if there exists a morphism $\pi:\,\Y\to\X$ over $S$ such that $\Y\to S$
is smooth, $\pi$ is an isomorphism outside of $\Sigma$, and $\P=\pi^{-1}(\Sigma)$ is a $\P^1$-bundle over $S$
such that all geometric fibers have
self-intersection~$-4$. In this case  $\X\to S$ satisfies the Koll\'ar condition.
\end{lemma}

\begin{proof}
 In one direction, we obtain $\Y$ by blowing up $\Sigma$. In the opposite direction, since the question is \'etale-local on $S$ and $\X$, we can assume that $\X$ and $S$
 are spectra of Henselian local rings.
 By \cite[Th. 2.13]{LN}, it suffices to find relative Cartier divisors $D_1$ and $D_2$ of $\X\to S$
 such that their scheme-theoretic intersections with $\P$ are disjoint sections of the $\P^1$-bundle.
 As in the proof of \cite[Th. 2.11]{LN}, their existence follows from surjectivity of $\Pic\X\to\Pic\P^1_s$ \cite[Cor. 21.9.12]{EGAIV}, where $s\in S$ is the closed point.
 Finally, $\bX_{S'}$ (being toric) and hence $\X$ satisfy the Koll\'ar condition.
\end{proof}

Recall that we have a contraction $Y\mathop{\to}\limits^\alpha S_0$ of $\Delta_1$ to a ${1\over 4}(1,1)$ singularity.

\begin{lemma}\label{asfvasfafsb}
We can ``blow down'' the deformation $\Y\to\Spec\cR$ of $Y$ to
the deformation $\bar\Y\to\Spec\cR$ of $S_0$, i.e. there exists a morphism $\Y\to\bar\Y$ of deformations over $\Spec\cR$
which on the special fiber gives $\alpha$.

This morphism contracts $\cC$ to a section $\Sigma$ of $\bar\Y\to\Spec\cR$ and it is an isomorphism outside $\Sigma$.
The family $\bar\Y\to\Spec\cR$ has a ${1\over 4}(1,1)$ singularity along $\Sigma$
and is smooth elsewhere. It is  $\Q$-Gorenstein.
\end{lemma}

\begin{proof}
 This follows from the fact that $R^1\alpha_*(\O_Y)=0$ as in \cite{W1} (where the equi-characteristic local case is worked out).
Specifically, let $\hat\Y$ be the formal completion of the special fiber in $\Y$.
Let $\hat{\bar\Y}$ be a formal scheme with underlying topological space $S_0$ and sheaf of rings $\alpha_*\O_{\hat\Y}$.
The vanishing of $R^1\alpha_*(\O_Y)$ implies that $\hat{\bar\Y}$ is flat over $\Spf\cR$ by
\cite[0.4.4]{W1}.
Since $H^2(S_0,\O_{S_0})=0$ and $S_0$ is projective, $\hat{\bar\Y}$ carries an ample line bundle,
and therefore is a formal fiber of a scheme $\bar\Y$ projective and flat over $\Spec\cR$,
by Grothendieck's existence theorem \cite[5.4.5]{EGAIII1}.
Since the formal fiber functor is fully-faithful \cite[5.4.1]{EGAIII1},
the morphism $\hat\Y\to\hat{\bar\Y}$ is induced by the morphism
$\alpha:\,\Y\to\bar\Y$. The rest follows from Lemma~\ref{kjbkhbk}.
\end{proof}

 \begin{lemma}\label{asvsfgsfg}
 The $\cR$-stack $\cK^\omega_\cR$ is smooth at $S_0\to\Spec k$
\end{lemma}

\begin{proof}
It suffices to prove that the special fiber of $\cK^\omega_\cR$, i.e. the algebraic stack of Koll\'ar families over $k$,
is smooth at $S_0\to\Spec k$. There are several ways to deduce this from Theorem~\ref{julie}.
One is to use the theory of index one covers as in \cite[Section 3]{Ha} (which assumes characteristic $0$ but in our case this is not important
because the index of the singularity $2$ is not divisible by the characteristic $7$).
One can also mimic calculations in \cite{Ha} in the setting of \cite{AH}.
Finally, one can apply \cite[Prop. 6.4]{W} (or \cite[Th. 4.6]{LN}), which shows that
the morphism of deformation functors of artinian rings $\Def X \to\Def^{loc}X$ is smooth
and that local $\Q$-Gorenstein deformations of a ${1\over4}(1,1)$-singularity are unobstructed.
\end{proof}

Let $(D; \Gamma, N)$ be the general fiber of the family $(\Y; \cC, \N)\to\Spec\cR$
after pull-back to $\Spec\C$.
Let $D\to D_0$ be the contraction of $\Gamma$.
Here $D_0$ is the general fiber of $\bar\Y\to\Spec\cR$ (after pull-back to $\Spec\C$).

\begin{theorem}\label{corollar}
There exists a $\Q$-Gorenstein family of complex surfaces $\bS\to U$ over a smooth irreducible complex curve
such that one of the fibers is $D_0$ and another fiber is the Craighero--Gattazzo surface $S$.
\end{theorem}

\begin{proof}
This follows from Theorem~\ref{abstract} and Lemmas~\ref{asfvasfafsb}, \ref{asfgasrgarg}, and \ref{asvsfgsfg}.
\end{proof}

\begin{figure}[htbp]
\includegraphics[width=9.3cm]{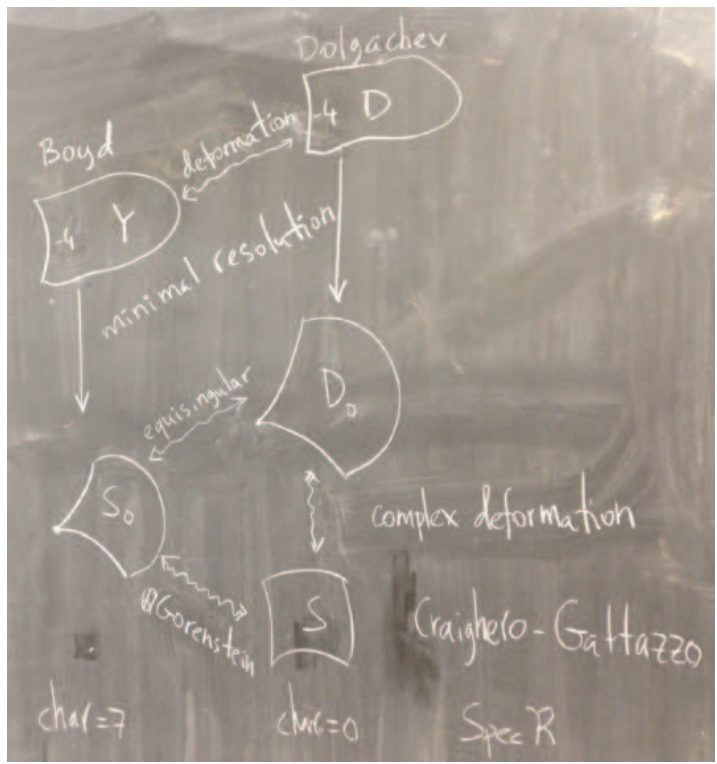}
\caption{The big picture}
\label{f2}
\end{figure}

The following corollary (of the proof) was first proved in \cite[Th.~0.31]{CP}.

\begin{porism}
The Craighero--Gattazzo surface is unobstructed and its local moduli space
is smooth of dimension $8$.
\end{porism}

\begin{proof}
Since the stack of Koll\'ar families $\cK^\omega_\cR$ is $\cR$-smooth at $S_0\to\Spec k$, the stack
$\cK^\omega_\C$ is $\C$-smooth at $S\to\Spec\C$. But in the neighborhood of a smooth surface such as $S$,
$\cK^\omega_\C$ can be identified with the Deligne--Mumford stack of Gieseker families of
canonically polarized surfaces with canonical singularities.
\end{proof}

%%%%%%%%%%%%%%%%%%%%%%%%%%%%%%%%%%%%%%%%%%%%%%%%%%%%%%%%%%%%%%%%%%%%%%%%%%%%%%%%%%%%%%%%%%%%%%%%%%%%%%%%%%%%%

\section{Calculation of the fundamental group}

\begin{proposition}\label{DolgTh}
The surface $D$ is a complex Dolgachev surface with multiple fibers of multiplicity $2$ and $3$.
In particular, $\pi_1(D)=1$. \label{trivialfun}
\end{proposition}

\begin{proof}
We first claim that
\begin{equation}\label{zfbfbzfbzfb}
\pi_1^{\text{alg}}(D_0)=1.
\end{equation}
We are going to use that $\pi_1^{\text{alg}}(S)=1$ \cite{DW}.
Since this is the only fact about $S$ that we need,
we can shrink the curve $U$ from Theorem~\ref{corollar} and without loss of generality assume that $U$ is a complex disc.
Since $\S$ contracts onto $D_0$, we have $\pi_1(\S)=\pi_1(D_0)$.
Now using the same argument as in \cite[p.601]{Xiao91},
we have an exact sequence
$$\pi_1(S) \to \pi_1(\S) \to \pi_1(U) \to 1,$$
and so $\pi_1(S)$ surjects onto $\pi_1(D_0)$.
The right exactness of profinite completions \cite[Prop.3.2.5]{RZ10}
implies that $\pi_1^{\text{alg}}(S)$ surjects onto $\pi_1^{\text{alg}}(D_0)$, which implies \eqref{zfbfbzfbzfb}.
Alternatively, surjectivity of $\pi_1^{\text{alg}}(S)\to\pi_1^{\text{alg}}(D_0)$ follows from the Grothendieck's specialization theorem  \cite[Cor. 2.3]{SGA1}.

We have $K_D^2=0$.
By Lemma~\ref{nefness} and Corollary~\ref{corollar}, $K_{D_0}$ is nef.
Therefore, $D$ is not rational. Indeed, if $D$ is rational, then by Riemann-Roch
$h^0(D,-K_D) \geq 1$ and so $-K_D \sim E \geq0$.
Since $K_D \cdot \Gamma = 2$, we have $\Gamma \subset E$.
We know that $f^*(2K_{D_0}) \sim -2E+\Gamma$ where $f \colon D \to D_0$ is the minimal resolution. But $E \neq \Gamma$, and so $f^*(2K_{D_0})$ cannot be nef.
Also, the Kodaira dimension of $D$ cannot be $0$ because of the Enriques classification
and $K_D \cdot \Gamma = 2$, and cannot be $2$ because of Kawamata's argument \cite{K92} (see \cite[Lemma 2.4]{R}).
Therefore the Kodaira dimension is $1$, and so $D$ is an elliptic fibration over $\P^1$ (since $q(D)=0$).

Say we have $r$ multiple fibers of multiplicities $m_1,\ldots,m_r$. By \cite[p.601]{Xiao91},
$$\pi_1(D)\simeq \langle a_1,\ldots,a_r : a_1\cdots a_r=a_1^{m_1}=\ldots=a_r^{m_r}=1 \rangle.$$
But this group is residually finite (see \cite[p.126]{LySch77} and \cite[p.141 last paragraph]{LySch77}).
We also have $\pi_1^{\text{alg}}(D)=\pi_1^{\text{alg}}(D_0)$
(see \cite{Ko93}), and so by the above we get $\pi_1(D)=1$.
This implies that there are only two multiple fibers $m_1 F_1, m_2 F_2$ with coprime multiplicities $m_1,m_2$. Let $F$ be a general fiber of $D \to \P^1$, and let $\Gamma \cdot F =d$. Then, since $K_D \sim -F + (m_1-1)F_1 + (m_2-1) F_2$, we have $\Gamma \cdot K_D= d- \frac{d}{m_1} - \frac{d}{m_2} =2$. In addition, since $\Gamma \cdot F_1=\frac{d}{m_1}$ and $\Gamma \cdot F_2=\frac{d}{m_2}$, we have $d=\lambda m_1 m_2$, and so $\lambda (m_1 m_2-m_1 -m_2)=2$. The only possible solutions, up to permuting $1$ and $2$, are $\lambda=2$, $m_1=2$, $m_2=3$.
\end{proof}

\begin{theorem}\label{finalresult}
$\pi_1(S)=1$.
\end{theorem}

\begin{proof}
Here we use the method of \cite{LP07}, which applies Van Kampen's Theorem and the Milnor fiber of the $\Q$-Gorenstein smoothing of $\frac{1}{4}(1,1)$. We only need $\pi_1(D \setminus \Gamma)=1$. By Van Kampen's theorem, we have $\pi_1(D) \simeq \pi_1(D \setminus \Gamma) / \overline{\langle \alpha \rangle}$ where $\alpha$ is a loop around $\Gamma$, and $\overline{\langle \alpha \rangle}$ is the smallest normal subgroup of $\pi_1(D \setminus \Gamma)$ containing $\langle \alpha \rangle$. We can and do consider $\alpha$ as given by a loop around $N$, since $N$ and $\Gamma$ intersect transversally. As $N \cdot \Gamma=1$, the set $N':=N \cap (D \setminus \Gamma)$ is simply-connected, and so $\alpha \subset N' \subset D \setminus \Gamma$ is homotopically trivial. Therefore $\overline{\langle \alpha \rangle}=1$, and so $\pi_1(D \setminus \Gamma)=1$ since by Proposition \ref{trivialfun} we have $\pi_1(D)=1$. After this, one directly applies \cite{LP07} (pages 493 and 499).
\end{proof}

%-----------------------------------------------------------------------------------------------------------
\section{Genus $2$ Lefschetz fibration on a Dolgachev surface}\label{LefSec}

In Section~\ref{DefSec} we constructed a lifting of the Boyd surface $Y$ (a Dolgachev surface in characteristic~$7$)
to some Dolgachev surface $D$ in characteristic~$0$.
Using results of Section~\ref{boydgeo1} we can be much more explicit:

\begin{theorem}\label{Camped}
The Boyd surface $Y$ can be lifted to a complex Dolgachev
surface $D$ of type $2,3$, which possesses an $I_4$ fiber of multiplicity $2$,
two $(-4)$-curves, and four elliptic $(-1)$-curves $E_1,\ldots,E_4$.
This surface has a Campedelli-type description as the minimal resolution of singularities of the double cover
of $\P^1\times\P^1$ with four elliptic singularities and two $A_1$ singularities.
\end{theorem}

\begin{proof}
In Section~\ref{boydgeo1}, the main point was to prove that $$H^2(\P_1,T_{\P_1}(-\log(\Delta+B_1+B_2)))=0.$$ By applying the $(-1)$ and $(-2)$ principles as before, we have $$H^2\big(\P_1,T_{\P_1}(-\log(\Delta+B_1+B_2 +\sum\bar N_i+ \sum\bar G_i + \sum\bar E_i + \sum\bar C_i + \sum \bar F_i))\big)=0.$$ By Lemma \ref{ajhsdbafhjsfb}, preservation of intersection numbers, and $H^2(\P_1,\O_{\P_1})=0$, we have that the configuration of curves $\Delta+B_1+B_2 +\sum\bar N_i+ \sum\bar G_i + \sum\bar E_i + \sum\bar C_i + \sum \bar F_i$ exists in $\P_1$ over $\C$. We will use the same notation as in char $7$. Then, by contracting $\sum\bar N_i+ \sum\bar G_i + \sum\bar E_i + \sum\bar C_i + \sum \bar F_i$, we obtain curves $\Delta+B_1+B_2$ in $\P_{\C}^1 \times \P_{\C}^1$ with the corresponding singularities. In this way, we can check that $\Delta \sim (1,1)$ and $B_i \sim (3,3)$ in Pic$(\P^1 \times \P^1)$. Notice that the two singularities of $B_1$ and the two singularities of $B_2$ may not be located at the special position we had in char $7$. Let us call these points $P_1,\ldots,P_4$ as before.

The linear system $|\O(2,2)|$ contains a member, which we call $\Gamma$,
that passes through $P_1,\ldots,P_4$ with the direction of the tangent cone to $B_1 \cup B_2$.
Indeed, $\Gamma$ exists because $h^0(\P^1\times\P^1, \O(2,2))=9$ and passing through
 $4$ points with $4$ given directions imposes $8$ conditions.

Then one easily checks in $\P_1$ that
\begin{equation}\label{mmm}
B_1 + B_2 + 2\bar N_1 + 2\bar N_2
\sim  3 \Gamma
\end{equation}
as well as
\begin{equation}
 B_1 + B_2 + \sum_{i=1}^4\bar G_i \sim 2 \Big( 3 \sigma^*(\Delta)- \bar N_1-\bar N_2-3\sum_{i=1}^4\bar E_i - \sum_{i=1}^4\bar G_i\Big).
 \label{mmmm}\end{equation}

In this way, \eqref{mmm} gives an elliptic fibration $\P_1 \to \P_{\C}^1$ with one multiple fiber $\Gamma$ of multiplicity $3$, and \eqref{mmmm} gives a double cover $W \to \P_1$ of $\P_1$ branched along $B_1 + B_2 + \sum_{i=1}^4\bar G_i$, just as before. Again the pre-images of $\bar G_1, \ldots, \bar G_4$ give $(-1)$-curves in $W$, which we contract to obtain a surface $D$. Using the standard formulas for double covers, as before, we get $K_D^2=0$ and $\chi(\O_D)=1$. Also, we can directly compute $p_g(D)=0$ using the defining line bundle of the double cover, and so $q(D)=0$.
The pull-back of the elliptic fibration $\P_1 \to \P_{\C}^1$ gives an elliptic fibration $D \to \P_{\C}^1$ with two multiple fibers: the pre-images of $\Gamma$ and $B_1+B_2+\bar N_1 + \bar N_2$, with multiplicities $3$ and $2$ respectively. The two $(-4)$-curves are pre-images of $\Delta$.
\end{proof}

%%%%%%%%%%%%%%%%%%%%%%%%%%%%%%%%%%%%%%%%%%%%%%%%%%%%%%%%%%%%%%%%%%%%%%%%%%%%%%%%%%%%%%%%%%%%%%%%%%%%%%%%%%%%%

Moreover, we notice that the pull-backs of the two rulings of $\P_{\C}^1 \times \P_{\C}^1$ give two distinct genus two fibrations $D \to \P_{\C}^1$. %These two fibrations are Lefschetz fibrations, because each of them comes from pulling back a ruling under a double cover branched along a $6$-section, and $Y$ is a minimal surface. The singular fibers contain the four $(-1)$-elliptic curves $E_1,\ldots,E_4$.

\begin{theorem}\label{Lefschetz}
There exist Dolgachev surfaces (with multiple fibers of multiplicity~$2,3$) which carry genus~$2$ Lefschetz fibrations, specifically genus~$2$ fibrations without multiple components in fibers and such that the only singularities of fibers are nodes.
\end{theorem}

\begin{proof}
%In the Godeaux case, we simply combine our Theorem \ref{finalresult} (the CG surface $S$ is simply-connected) with \cite[Prop. 3.2]{DW},
%where Dolgachev and Werner showed the existence of a genus~$2$ Lefschetz fibration on $S$ with $15$ singular fibers, all of them of the same type: two elliptic curves of self-intersection $(-1)$ intersecting at one nodal point (see analysis of fibers in \cite[Sec. 3]{DW}, notice that cases (ii) and (iii) do not occur by
%\cite[Th. 6.2]{DW}).

%Next we study the case of Dolgachev surfaces.
In characteristic $7$, we have two genus two fibrations on the Boyd surface induced by the two rulings in $\P^1 \times \P^1$. We first want to find out the singular fibers of these fibrations. For that, we need to look at the induced morphisms $B_i \subset \P \to \P^1 \times \P^1 \to \P^1$ for each $i$ and for each ruling.

Using the equations \eqref{b1eqqqq} and \eqref{b2eqqqq} of $B_1$ and $B_2$ respectively, we obtain that, for the ruling $\beta=x/t$, the morphism $B_1 \to \P^1$ has branch points at $\beta$ satisfying $(\beta^2+1)^2=0$, and the morphism $B_2 \to \P^1$ is branched at $\beta$ satisfying $\beta^4+4 \beta^2+1=0$. One verifies that in the first case, the points of ramification are $Q_1=(-i,i)$ and $Q_2=(i,-i)$, and $B_1$ is tangent to the ruling with flex points at $Q_1$ and $Q_2$. For the second ruling, the roles of $B_1$ and $B_2$ are interchanged in relation to ramification, and $B_2$ is tangent to the ruling with flex points at $Q_1$ and $Q_2$ for $B_2$.

Using the previous observations on the ramification points of $B_1 \to \P^1$ and $B_2 \to \P^1$, we obtain the following singular fibers for the genus~$2$ fibrations $Y \to \P^1$ (we take it from one ruling, the other is analogous):

\begin{itemize}
\item[(1)] two reduced singular fibers consisting of $E_1 \cup A_1 \cup E_4$ and $E_2 \cup A_2 \cup E_3$ where $E_i$ are disjoint elliptic $(-1)$-curves, and $A_i$ are $(-2)$ rational curves, each intersecting two $E_j$ at one nodal point.

\item[(2)] two reduced singular fibers over $\beta=i,-i$ consisting of one nodal rational curve together with $N_1$, and another rational nodal curve with $N_2$. Each of the $N_i$ passes through the corresponding node, forming a simple triple point for the fiber.

\item[(3)] four reduced singular curves, each consisting of a nodal curve whose resolution is an elliptic curve.
\end{itemize}

We claim that there exists a lifting of this Dolgachev surface to characteristic $0$ as in Theorem~\ref{Camped} such that case (2) is eliminated.
In other words, we have to construct a lifting
of $\P_1$ together with the curves $\Delta+B_1+B_2 +\sum\bar N_i+ \sum\bar G_i + \sum\bar E_i + \sum\bar C_i + \sum \bar F_i$ such that the flex ramification points for $B_1 \to \P^1$ disappear, becoming simple ramification for a degree $3$ morphism $B_1 \to \P_{\C}^1$. Using the Macaulay2 code in the Appendix, we show the existence of a first other deformation of that type. This together with unobstructed deformations, as in the remark above, gives a lifting to $\Spec\cR$
such that, over the generic point, the curve $B_1$ is not flex with respect to any ruling. In this way, at least for one ruling, the corresponding genus~$2$ fibration on the complex $2,3$ Dolgachev surface has only singular fibers which are reduced and with nodes as singularities, i.e. it is a Lefschetz fibration.
\end{proof}

\section*{APPENDIX}

This appendix contains the Macaulay2 source code used to compute the rank of matrices in the proof of Theorem~\ref{julie}
and Theorem~\ref{Lefschetz}.
\footnotesize{
\begin{verbatim}

--For simplicity, this includes the extra variable y.
R=ZZ/7[t,c1,c2,c3,c4,d1,d2,d3,d4,a20,a21,a31,a02,a12,a03,b10,b11,b22,b32,b23
,x,y];

--Adjoin a square root of -1:
R1=R/(t^2+1)

-- Twenty-one of the restrictions on coefficients arising from forcing desired
singularities at the points to which P1, P2, P3, P4 deform. These allow us to
reduce the number of variables from 40 to 19.

a33=0; b33=d1-3*c1; a32=-3*c1-6*d1; a23=2*c1-3*d1;
a22=2*a31+4*b23-2*b32; a13=2*a31-2*b32+4*b23;
a30=2*c1-d1+3*a12+2*a21+a03; b30=0; b31=3*c2+2*d2; b20=3*d2+c2;
b21=6*b10+6*a31+3*a20; b00=-c2-2*d2-2*b11-4*b22-b33;
b03=0; b02=2*d3+3*c3; b13=3*d3+c3; b01=5*a13+2*b23+3*a02;
b12=4*a13+6*b23+a02; a00=0; a01=3*c4+6*d4; a10=3*d4-2*c4;
a11=a02+4*b01+6*b10;

--The intersection of $\bar{B}_1$ and $\bar{B}_2$ with $\Delta$
(Here, x=$\beta$):

g1bar= (1+x)^3*(a00+a01*x+a02*x^2+a03*x^3)
+(1+x)^2*(1-x)*(a10+a11*x+a12*x^2+a13*x^3)
+(1+x)*(1-x)^2*(a20+a21*x+a22*x^2+a23*x^3)
+(1-x)^3*(a30+a31*x+a32*x^2+a33*x^3);

g2bar= (1+x)^3*(b00+b01*x+b02*x^2+b03*x^3)
+(1+x)^2*(1-x)*(b10+b11*x+b12*x^2+b13*x^3)
+(1+x)*(1-x)^2*(b20+b21*x+b22*x^2+b23*x^3)
+(1-x)^3*(b30+b31*x+b32*x^2+b33*x^3);

--The derivatives of g1bar and g2bar:
dg1bar= diff(x, g1bar);

dg2bar=diff(x, g2bar);

-- B1 and B2 pass through Q1 and Q2:
B1Q1=sub(g1bar, x=>t); B1Q2=sub(g1bar, x=>-t);

B2Q1=sub(g2bar, x=>t); B2Q2=sub(g2bar, x=>-t);

-- B1 passes through Q3 (x=-2+4i), Q4 (x=-2-4i):
B1Q3=sub(g1bar, x=>-2+4*t); B1Q4=sub(g1bar, x=>-2-4*t);

-- B2 passes through Q5, Q6;
B2Q5=sub(g2bar, x=>-3+5*t); B2Q6=sub(g2bar, x=>-3-5*t);

-- B1 is tangent at Q3, Q4:
dB1Q3=sub(dg1bar, x=>-2+4*t); dB1Q4=sub(dg1bar, x=>-2-4*t);

-- B2 tangent at Q5, Q6
dB2Q5=sub(dg2bar, x=>-3+5*t); dB2Q6=sub(dg2bar, x=>-3-5*t);

-- Each of the following ideals gives the kernel of one of the seven the
systems of linear equations. Notice that in each, we include the
remaining seven restrictions arising from forcing desired singularities
at the points to which P1, P2, P3, P4 deform.

--move Q4 off Delta, moving Q5, Q6 along, keeping tangent direction
at Q3, Q5, Q6

I1=ideal(a30+c2+3*d2, b32-2*b10-4*a31-2*a20,  a03-c3-3*d3,
c3+2*d3-3*b11+b33+2*b22+b00, b00-d4+3*c4,
a20-2*a02-2*b01-3*b10,4*c4+5*d4+2*a03+3*a12+a21+5*a30,
B1Q1, B1Q2, B2Q1, B2Q2, B1Q3, dB1Q3, B1Q4-1,B2Q5, B2Q6);

--move Q5 off Delta, moving Q4, Q6 along, keeping tangent direction
at Q3, Q4, Q6:

I2=ideal(a30+c2+3*d2, b32-2*b10-4*a31-2*a20,  a03-c3-3*d3,
c3+2*d3-3*b11+b33+2*b22+b00, b00-d4+3*c4,
a20-2*a02-2*b01-3*b10,4*c4+5*d4+2*a03+3*a12+a21+5*a30,
B1Q1, B1Q2, B2Q1, B2Q2, B1Q3, dB1Q3, B1Q4, B2Q6,B2Q5-1) ;

--move Q6 off Delta, moving Q4, Q5 along, keeping tangent direction
at Q3, Q4, Q5:

I3=ideal(a30+c2+3*d2, b32-2*b10-4*a31-2*a20,  a03-c3-3*d3,
c3+2*d3-3*b11+b33+2*b22+b00, b00-d4+3*c4,
a20-2*a02-2*b01-3*b10,4*c4+5*d4+2*a03+3*a12+a21+5*a30,
B1Q1, B1Q2, B2Q1, B2Q2, B1Q3, dB1Q3, B1Q4,B2Q5, B2Q6-1)


--leave all points on Delta, moving Q4, Q5, Q6 along, change tangent
direction at Q3:

I4=ideal(a30+c2+3*d2, b32-2*b10-4*a31-2*a20,  a03-c3-3*d3,
c3+2*d3-3*b11+b33+2*b22+b00, b00-d4+3*c4,
a20-2*a02-2*b01-3*b10,4*c4+5*d4+2*a03+3*a12+a21+5*a30,
B1Q1, B1Q2, B2Q1, B2Q2, B1Q3, B1Q4, dB1Q3-1,B2Q5, B2Q6)

--leave all points on Delta, moving Q5, Q6 along, change tangent
direction at Q4, keep tangent direction at Q3:

I5=ideal(a30+c2+3*d2, b32-2*b10-4*a31-2*a20,  a03-c3-3*d3,
c3+2*d3-3*b11+b33+2*b22+b00, b00-d4+3*c4,
a20-2*a02-2*b01-3*b10,4*c4+5*d4+2*a03+3*a12+a21+5*a30,
B1Q1, B1Q2, B2Q1, B2Q2, B1Q3, B1Q4, dB1Q4-1,B2Q5, B2Q6,
dB1Q3)

--leave all points on Delta, moving Q4, Q6 along, change tangent
direction at Q5, keep tangent direction at Q3:

I6=ideal(a30+c2+3*d2, b32-2*b10-4*a31-2*a20,  a03-c3-3*d3,
c3+2*d3-3*b11+b33+2*b22+b00, b00-d4+3*c4,
a20-2*a02-2*b01-3*b10,4*c4+5*d4+2*a03+3*a12+a21+5*a30,
B1Q1, B1Q2, B2Q1, B2Q2, B1Q3, B1Q4, dB2Q5-1,B2Q5, B2Q6,
dB1Q3)

--leave all points on Delta, moving Q4, Q5 along, change tangent
direction at Q6, keep tangent direction at Q3:

I7=ideal(a30+c2+3*d2, b32-2*b10-4*a31-2*a20,  a03-c3-3*d3,
c3+2*d3-3*b11+b33+2*b22+b00, b00-d4+3*c4,
a20-2*a02-2*b01-3*b10,4*c4+5*d4+2*a03+3*a12+a21+5*a30,
B1Q1, B1Q2, B2Q1, B2Q2, B1Q3, B1Q4, dB2Q6-1,B2Q5, B2Q6,
dB1Q3)

-- Check the dimension of each ideal (note: each has one less
dimension that Macaulay2 gives, because of the extra variable y)

-- four-dimensional
dim1=dim(I1); dim2=dim(I2); dim3=dim(I3); dim4=dim(I4);

--three-dimensional:
dim5=dim(I5); dim6=dim(I6); dim7=dim(I7);

-- The remaining code is used to prove the existence of the Lefschetz
fibration. In particular, we prove existence of a deformation of B1+B2
so that B1+B2 maintains its singularities at P1,..., P4 and so that B1
is no longer tangent to the fiber x=i or x=-i at Q1, Q2:

-- $\bar{B}_1$ and $\bar{B}_2$ (alpha=y, beta=x):
g1bar= (a00+a01*x+a02*x^2+a03*x^3)
+y*(a10+a11*x+a12*x^2+a13*x^3)
+y^2*(a20+a21*x+a22*x^2+a23*x^3)
+y^3*(a30+a31*x+a32*x^2+a33*x^3);

g2bar=(b00+b01*x+b02*x^2+b03*x^3)
+y*(b10+b11*x+b12*x^2+b13*x^3)
+y^2*(b20+b21*x+b22*x^2+b23*x^3)
+y^3*(b30+b31*x+b32*x^2+b33*x^3);

-- Writing the local equation of B1 along the fiber at Q1 and Q2:
B1Q1bar=sub(sub(g1bar, x=>t), y=>y-t);
B1Q2bar=sub(sub(g1bar, x=>-t), y=>y+t);

-- These ensure B1 vanishes at Q1 and Q2:

van1=sub(B1Q1bar, y=>0); van2=sub(B1Q2bar, y=>0);

-- These (when nonzero) force B1 to be no longer tangent to the fiber
x=i, x=-i at Q1, Q2:

dB1Q1=diff(y, B1Q1bar); dB1Q2=diff(y, B1Q2bar);


Lefschetz=ideal(a30+c2+3*d2, b32-2*b10-4*a31-2*a20,
a03-c3-3*d3, c3+2*d3-3*b11+b33+2*b22+b00, b00-d4+3*c4,
a20-2*a02-2*b01-3*b10,4*c4+5*d4+2*a03+3*a12+a21+5*a30,
van1, van2, dB1Q1-1, dB1Q2-1)

dimL= dim(Lefschetz); -- 10-dimensional
\end{verbatim}
}

%----------------------------------------------------------------------------------------------------------------------------------------------


\begin{thebibliography}{99}


\bibitem[AH]{AH}
    D. Abramovich, B. Hassett,
    \emph{Stable varieties with a twist},
    Classification of algebraic varieties, proceedings of the 2009 conference at Schiermonnikoog the Netherlands,
    Carel Faber, Gerard van der Geer, and Eduard J.N. Looijenga eds., 1-38, European Mathematical Society,
    Z\"urich, 2011.

\bibitem[Ar]{Ar}
    M.~Artin,
    \emph{Algebraization of formal moduli. I.},
    Global Analysis (Papers in Honor of K. Kodaira), pp. 21--71 Univ. Tokyo Press, Tokyo, 1969.

\bibitem[Ar1]{Ar1}
    M.~Artin,
    \emph{Versal deformations and algebraic stacks},
    Inventiones Math., {\bf 27}, (1974), Issue 3, pp. 165--189.

\bibitem[AK]{AK}
    T. Ashikaga, K. Konno,
    \emph{Examples of degenerations of Castelnuovo surfaces},
    J. Math. Soc. Japan {\bf 43 } (1991), no. 2, 229--246.

\bibitem[B]{B01}
    L. B$\check{a}$descu,
    \emph{Algebraic Surfaces},
    Universitext, Springer-Verlag, Berlin, 2001.

\bibitem[Ba]{Ba}
    R. Barlow,
    \emph{A simply connected surface of general type with $p_g = 0$.},
    Invent. Math. {\bf 79} (1985), no. 2, 293--301.

\bibitem[BHPV]{BHPV04}
    W. P. Barth, K. Hulek, C. A. M. Peters, A. Van de Ven,
    \emph{Compact complex surfaces},
    Ergebnisse der Mathematik und ihrer Grenzgebiete. 3. Folge., second edition, vol. 4, Springer-Verlag, Berlin, 2004.

\bibitem[BCP]{BCP}
    I. Bauer, F. Catanese, R. Pignatelli,
    \emph{Surfaces of general type with geometric genus zero: a survey},
    Complex and Differential Geometry. Springer Proceedings in Mathematics, {\bf 8} (2011), 1--48.

\bibitem[BK]{BK}
    I. Baykur, M. Korkmaz,
    \emph{Small Lefschetz fibrations and exotic 4-manifolds},
    arXiv:1510.00089.

\bibitem[CG]{CG}
    P.C. Craighero, R. Gattazzo,
    \emph{Quintic surfaces of $\P^3$ having a nonsingular model with $q = p_g = 0$, $P_2\ne 0$},
    Rend. Sem. Mat. Univ. Padova {\bf 91} (1994), 187--198.

\bibitem[CL]{CL}
    C. LeBrun, F. Catanese,
    \emph{On the scalar curvature of Einstein manifolds},
    Math. Res. Lett. {\bf 4}, no.6 (1997), 843--854.

\bibitem[CP]{CP}
    F. Catanese, R. Pignatelli,
    \emph{On simply connected Godeaux surfaces},
    Complex analysis and algebraic geometry, de Gruyter Berlin (2000), 117--153.

\bibitem[CD]{CD}
    F. R. Cossec, I. Dolgachev,
    \emph{Enriques surfaces I},
    Progress in Mathematics, 76. Birkh$\ddot{a}$user Boston, Inc., Boston, MA,  1989.

\bibitem[DW]{DW}
    I. Dolgachev, C. Werner,
    \emph{A simply connected numerical Godeaux surface with ample canonical class},
    J. Algebraic Geom. {\bf 8} (1999), no.4, 737--764. Erratum: J. Algebraic Geom. {\bf 10} (2001), no.2, 397.

\bibitem[FS]{FS}
    R. Fintushel, R. Stern,
    \emph{Families of simply connected 4-manifolds with the same Seiberg-Witten invariants},
    Topology {\bf 43} (2004), no. 6, 1449--1467.

\bibitem[Fr]{Fr}
    M. Freedman,
    \emph{The topology of four-dimensional manifolds},
    Journal of Differential Geometry {\bf 17} (1982), no. 3, 357--453.

\bibitem[EGAIII1]{EGAIII1}
    A. Grothendieck, J. Dieudonn\'e,
    \emph{El\'ements de g\'eom\/etrie alg\'ebrique: III. \'Etude cohomologique des faisceaux coh\'erents, Premi\`ere partie},
    Publ. Math. de l'IH\'ES {\bf 11} (1961), 5--167.

\bibitem[EGAIV]{EGAIV}
    A. Grothendieck, J. Dieudonn\'e,
    \emph{El\'ements de g\'eom\'etrie alg\'ebrique: IV. \'Etude locale des sch\'emas et des morphismes de sch\'emas, Quatri\'eme partie},
    Publ. Math. de l'IH\'ES {\bf 32} (1966), 5--361.

\bibitem[Ha]{Ha}
    P.~Hacking,
    \emph{Compact moduli of plane curves},
    Duke Math. J. {\bf 124} (2004), no. 2, 213--257.

\bibitem[K92]{K92}
    Y. Kawamata,
    \emph{Moderate degenerations of algebraic surfaces},
    Complex algebraic varieties (Bayreuth, 1990), 113--132, Lecture Notes in Math., 1507, Springer, Berlin, 1992.

\bibitem[Ko]{Ko93}
    J. Koll\'ar,
    \emph{Shafarevich maps and plurigenera of algebraic varieties},
    Invent. Math. {\bf 113} (1993), 177--215.

\bibitem[KSB]{KSB}
    J. Koll\'ar, N. I. Shepherd-Barron,
    \emph{Threefolds and deformations of surface singularities},
    Invent. Math. {\bf 91} (1988), 299--338.

\bibitem[L]{L}
    Y. Lee, \emph{Interpretation of the Deformation Space of a Determinantal Barlow Surface via Smoothings},
    Proc. of the AMS {\bf 130} (2002), no. 4, 963--969.

\bibitem[LN]{LN}
    Y. Lee, N. Nakayama,
    \emph{Simply connected surfaces of general type in positive characteristic via deformation theory},
    Proceedings of the London Mathematical Society {\bf 106} (2013), 225--286.

\bibitem[LP]{LP07}
    Y. Lee, J. Park,
    \emph{A simply connected surface of general type with $p_g=0$ and $K^2=2$},
    Invent. Math. {\bf 170} (2007), 483--505.

\bibitem[LS]{LySch77}
    R. C. Lyndon, P. E. Schupp,
    \emph{Combinatorial Group Theory},
    Springer-Verlag (1977).

\bibitem[Mum]{Mum}
    D.~Mumford,
    \emph{Abelian varieties},
    Tata Institute of fundamental research, Bombay, 1970.

\bibitem[PPSa]{PPSa}
    H. Park, J. Park, D. Shin,
    \emph{A simply connected surface of general type with $p_g = 0$ and $K^2 = 3$},
    Geom. Topol. {\bf 13} (2009), no. 2, 743--767.

\bibitem[PPSb]{PPSb}
    H. Park, J. Park, D. Shin,
    \emph{A simply connected surface of general type with $p_g = 0$ and $K^2 = 4$},
    Geom. Topol. {\bf 13} (2009), no. 3, 1483--1494.

\bibitem[PSU]{PSU}
    H. Park, D. Shin, G. Urz\'ua,
    \emph{A simply connected numerical Campedelli surface with an involution},
    Mathematische Annalen {\bf 357} (2013), no. 1, 31--49.

\bibitem[Pi]{Pi}
    H. Pinkham,
    \emph{Deformations of cones with negative grading},
    J. of Algebra 30 (1974), 92--102.

\bibitem[R]{R}
    J. Rana,
    \emph{A boundary divisor in the moduli space of stable quintic surfaces},
    arXiv:1407.7148, 58 pages.

\bibitem[RZ10]{RZ10}
    L. Ribes, P. Zalesskii,
    \emph{Profinite groups},
    Ergebnisse der Mathematik und ihrer Grenzgebiete. 3. Folge. A Series of Modern Surveys in Mathematics. v.40. Second edition. Springer-Verlag, Berlin, 2010.

\bibitem[SU]{SU}
    A. Stern, G. Urz\'ua,
    \emph{KSBA surfaces with elliptic quotient singularities, $\pi_1=1$, $p_g=0$, and $K^2=1,2$},
    arXiv:1409.4985, to appear in the Israel Journal of Mathematics.

\bibitem[Sch]{Sch}
    M. Schlessinger,
    \emph{Functors of Artin rings},
    Trans. Amer. Math. Soc. {\bf 130} (1968),  208--222.

\bibitem[SGA1]{SGA1}
    A. Grothendieck, M. Raynaud,
    \emph{Rev\^etements Etales et Groupe Fondamental: S\'eminaire de G\'eom\'etrie Alg\'ebrique de Bois-Marie 1960/61},
    Lecture Notes in Mathematics, vol. {\bf 224}, Springer--Verlag, 1971.

\bibitem[W]{W}
    J.~Wahl,
    \emph{Smoothings of normal surface singularities},
    Topology {\bf 20} (1981), 219--246.

\bibitem[W1]{W1}
    J.~Wahl,
    \emph{Equisingular deformations of normal surface singularities},
    Ann. of Math., 104 (1976), 325--356.

\bibitem[X]{Xiao91}
    G. Xiao,
    \emph{{$\pi\sb 1$} of elliptic and hyperelliptic surfaces},
    Internat. J. Math. {\bf 2} (1991), no. 5, 599--615.

\end{thebibliography}
\end{document}